\documentclass[english,11pt]{smfart}

\usepackage{etex}

\usepackage{amsbsy}
\usepackage{amsmath,amsfonts,amssymb,amsthm,mathrsfs,mathtools}

\usepackage{bm}

\usepackage[a4paper,vmargin={3cm,3cm},hmargin={3.5cm,3.5cm}]{geometry}
\usepackage[font=sf, labelfont={sf,bf}, margin=1cm]{caption}
\usepackage{graphicx}
\usepackage{epsfig}
\usepackage{latexsym}
\usepackage{xcolor}
\usepackage{ae,aecompl}
\usepackage{soul,framed}
\usepackage{comment}

\usepackage{xcolor}
\usepackage[pdfpagemode=UseNone,bookmarksopen=false,colorlinks=true,urlcolor=blue,citecolor=blue,citebordercolor=blue,linkcolor=blue]{hyperref}
\usepackage{smfhyperref}
\usepackage[capitalize]{cleveref}

\usepackage{pstricks}
\usepackage{enumerate}
\usepackage{tikz,animate,media9}						
\usepackage{todonotes}
\usepackage{pifont}
\usepackage{bm,marvosym}
\usepackage{algorithm}
\usepackage{algorithmic}

\usepackage{bbm}

\usepackage{tcolorbox}

\definecolor{aleacolor}{rgb}{0.16,0.59,0.78}

\hypersetup{
	breaklinks,
	colorlinks=true,
	linkcolor=aleacolor,
	urlcolor=aleacolor,
	citecolor=aleacolor}

\newcount\colveccount
\newcommand*\colvec[1]{
	\global\colveccount#1
	\begin{pmatrix}
		\colvecnext
	}
	\def\colvecnext#1{
		#1
		\global\advance\colveccount-1
		\ifnum\colveccount>0
		\\
		\expandafter\colvecnext
		\else
	\end{pmatrix}
	\fi
}

\newcommand{\ndN}{\mathbb{N}}

\newcommand{\ndR}{\mathbb{R}}
\newcommand{\ndC}{\mathbb{C}}

\renewcommand{\Pr}[1]{\mathbb{P}(#1)}

\newcommand{\Prb}[1]{\mathbb{P}\left(#1\right)}

\newcommand{\Ex}[1]{\mathbb{E}[#1]}

\newcommand{\Exb}[1]{\mathbb{E}\left[#1\right]}

\newcommand{\Va}[1]{\mathbb{V}[#1]}

\newcommand{\one}{{\mathbbm{1}}}

\newcommand{\convd}{\,{\buildrel \mathrm{d} \over \longrightarrow}\,}

\newcommand{\convp}{\,{\buildrel \mathrm{p} \over \longrightarrow}\,}

\newcommand{\eqdist}{\,{\buildrel \mathrm{d} \over =}\,}

\newcommand{\Poch}[2]{\bigl(#1\bigr)_{#2}}
\newcommand{\GPoch}[3]{\bigl(#1 \!\mid\! #2\bigr)_{#3}}

\newcommand{\cA}{\mathcal{A}}

\newcommand{\cC}{\mathcal{C}}

\newcommand{\cE}{\mathcal{E}}

\newcommand{\cG}{\mathcal{G}}
\newcommand{\cH}{\mathcal{H}}

\newcommand{\cO}{\mathcal{O}}
\newcommand{\cP}{\mathcal{P}}

\newcommand{\cS}{\mathcal{S}}
\newcommand{\cT}{\mathcal{T}}
\newcommand{\cU}{\mathcal{U}}

\newcommand{\cW}{\mathcal{W}}

\newcommand{\mC}{\mathsf{C}}

\newcommand{\mG}{\mathsf{G}}
\newcommand{\mH}{\mathsf{H}}

\newcommand{\mO}{\mathsf{O}}

\newcommand{\mR}{\mathsf{R}}
\newcommand{\mS}{\mathsf{S}}
\newcommand{\mT}{\mathsf{T}}

\newcommand{\pat}{\mathrm{pat}}

\newtheorem{theorem}{Theorem}[section]

\newtheorem{corollary}[theorem]{Corollary}
\newtheorem{proposition}[theorem]{Proposition}
\newtheorem{lemma}[theorem]{Lemma}

\newtheorem{definition}[theorem]{Definition}

\numberwithin{equation}{section}

\keywords{Poisson-Dirichlet processes, graphons, permutons}

\title{\textbf{Poisson-Dirichlet graphons and permutons}}
\date{}

\author{Benedikt Stufler}
\address[Benedikt Stufler]{Vienna University of Technology}
\email{benedikt.stufler at tuwien.ac.at}

\begin{document}

\vspace {-0.5cm}

\begin{abstract}
	We introduce classes of supergraphs and superpermutations with novel universal graphon and permuton limiting objects whose construction involves the two-parameter Poisson-Dirichlet process introduced by Pitman and Yor~(1997). We demonstrate the universality of these limiting objects through general invariance principles in a heavy-tailed regime and establish a  comprehensive phase diagram for the asymptotic shape of superstructures.  
\end{abstract}


\maketitle

\section{Introduction}

A recurring theme in modern probability is that seemingly unrelated discrete models admit common  universal limiting  objects.  Classic examples include invariance principles for partial sums, leading to Brownian motion, and scaling limits for random trees, leading to the Brownian continuum random tree introduced by Aldous~\cite{MR1085326}. In this  paper we identify a new family of universal limit objects that we call Poisson--Dirichlet graphons and permutons. We show that they arise universally when large random graphs and permutations are assembled by a composition (or substitution) mechanism in a heavy-tailed regime. 

It is well known that  many discrete structures admit decompositions into substructures. Examples include graphs built from modules, permutations built from substitution blocks, or trees built using grafting rules. 
Motivation for the mathematical study of these objects comes from data structures in computer science. Efficient structures often organise information by partitioning it into smaller parts and implementing rules to navigate between and within these parts~\cite{MR1370950}.

To formalise this systematically, we work with \emph{supergraphs} and \emph{superpermutations}: an object is specified by a \emph{head} (encoding inter-class relations) together with a family of \emph{components} (encoding intra-class relations), and we retain enough rooting information to make the decomposition unique. For graphs, each component governs adjacency within its vertex class, while the head governs adjacency between different classes.  For permutations, the component orders the elements inside a block, while the head permutation orders the blocks relative to each other. We assign non-negative weights to the head and component structures, and define a random superstructure of a given size $n$ generated with probability proportional to the product of these weights.

Our focus is on a heavy-tailed regime in which randomly generated superstructures typically exhibit a large number of macroscopic components. The limiting mass distribution of  component sizes is given by the two-parameter Poisson-Dirichlet point process introduced by Pitman and Yor~\cite{MR1434129}. We recall relevant properties of this process in Section~\ref{sec:twoparamproc}, and  determine  component size asymptotics, stated in Theorem~\ref{te:gibbsdilutegene}, in Section~\ref{sec:gibbs}. Our proof makes use of local large deviation asymptotics and renewal theorems  by Doney~\cite{MR1440141}. 


We proceed to study randomly generated superstructures in this regime through the lens of modern limit theories for dense graphs and permutations. 

The  theory of  graph limits was developed in a series of papers, see~\cite{MR2455626, MR2925382, MR2277152, MR2249277, MR2257396, MR2226430, MR2274085, MR2314569, MR2368030, MR2478356, MR2733066}.  Limiting objects, called graphons, are symmetric measurable functions $W: [0,1]^2 \to [0,1]$. They encode asymptotic subgraph densities. We summarise the required background of this  theory in Section~\ref{sec:graphons} following the exposition in~\cite{MR2463439}. 

In Section~\ref{sec:supergraphs} we introduce the Poisson-Dirichlet graphon $W_{\mathrm{PD}}(\alpha, \theta, L_{\cH}, L_{\cC})$, see Definition~\ref{def:wg}, as a random graphon constructed from a two-parameter Poisson-Dirichlet process $\mathrm{PD}(\alpha, \theta)$ and the laws $L_{\cH}$ and $L_{\cC}$ of two random graphons that represent the continuous analogues of head and component structures. Roughly speaking, the Poisson-Dirichlet graphon is constructed by taking an infinite number of independent samples from $L_{\cC}$, rescaling them according to the points of a ${\mathrm{PD}}(\alpha, \theta)$ process, placing them on the diagonal of  the unit square $[0,1]^2$, and determining adjacency between them according to the random graph on countably infinitely many vertices created from a single instance of a random graphon with law $L_{\cH}$. We characterise the marginal distributions of the Poisson-Dirichlet graphon using a two-parameter Chinese-restaurant process.

Our main invariance principle, stated in Theorem~\ref{te:main1}, shows distributional convergence of random supergraphs in a heavy-tailed \emph{dilute} regime towards Poisson-Dirichlet graphons. In this regime, we assume that the partition function $c_n$ of components is up to an exponential tilt regularly varying with index $-1-\beta$ for some $-\infty < \beta < 1$ and the  partition function $h_n$ of head-structures satisfies that after an exponential tilting, their tail sums vary regularly with index $-\alpha$ for $0<\alpha<1$. If isolated random head structures admit a random graphon with law $L_{\cH}$ as graph limit, and an isolated single large component admits a random graphon with law $L_{\cC}$, then in this dilute regime the limiting graphons combine to the Poisson-Dirichlet $W_{\mathrm{PD}}(\alpha, \theta, L_{\cH}, L_{\cC})$ for $\theta = -\alpha \beta$. We conclude with a comprehensive phase diagram of \emph{repellent} regimes that exhibit contrasting behaviour. Specifically, if the partition functions determine that we are in the \emph{dense} regime, the limiting law is given by~$L_\cH$, see Theorem~\ref{te:dense}. In the \emph{condensation} regime, the limiting law is given by~$L_{\cC}$, see Theorem~\ref{te:condensation}. And in the \emph{mixture} regime, the limiting law is a non-degenerate mixture of the laws $L_\cH$ and $L_{\cC}$, see Theorem~\ref{te:mixture}.

The construction of the Poisson-Dirichlet graphon may be iterated by setting $L_\cH$ or $L_\cC$ to the law of a Poisson-Dirichlet graphon. These multidimensional constructions occur  ``in the wild'' as limits of concrete combinatorial models of random graphs. 

We provide an example related to the prominent class $\cO$ of complement reducible graphs, also called cographs. This class may be characterised as the smallest class of finite simple graphs that contains the graph with one vertex and is closed under taking complements and forming finite disjoint unions. It was studied from enumerative~\cite{MR2154567} and algorithmic~\cite{MR2443114} viewpoints and it has some relevance in computer science, as certain hard computational problems may be solved in polynomial time if the input is restricted to cographs~\cite{MR3942334}. A uniform random complement reducible graph with a given number $n$ of vertices converges in distribution towards a random graphon called the Brownian graphon~\cite[Thm. 1]{zbMATH07749496}. An independent proof via different methods is given in~\cite{zbMATH07749513}. The Brownian graphon may be constructed from a  Brownian excursion~\cite[Sec. 4.2]{zbMATH07749496}, and its marginal distributions are related to the finite dimensional marginals of the Brownian tree. We recall relevant background for these marginals in Section~\ref{sec:browtre}.  

In Section~\ref{sec:corsu} we introduce classes of multidimensional iterated cographs. Specifically, we set $\cO_1 = \cO$, and for each $d \ge 2$ we let $\cO_d$ denote the class of supergraphs with head structures from $\cO$ and component structures from the class of ordered pairs of a marked vertex and a graph from $\cO_{d-1}$ (with no edges between the marked vertex and its corresponding graph from $\cO_{d-1}$).  This way, the  marked vertices in the components form a coarse ``prototype'' with the adjacency relation inherited precisely from the head structure. We call graphs from $\cO_d$ $d$-dimensional complementary reducible graphs with a prototype. We determine that their number  asymptotically grows like
\[
	 \frac{2^{-d}}{\Gamma(1- 2^{-d})} (2 \sqrt{2\log(2)-1})^{2-2^{-(d-1)}} n^{-1-2^{-d}} (2\log(2)-1)^{-n} n!.
\]
With $L_1$ denoting the law of the Brownian graphon $W_{\mathrm{Brownian}}$, we show in Theorem~\ref{te:main2} that a uniform random $d$-dimensional cograph with a prototype $\mO_{n,d}$ admits a limiting random graphon whose law $L_d$ is that of an iterated Poisson-Dirichlet graphon, specifically
\[
\mO_{n,d} \convd W_{\mathrm{PD}}(2^{-(d-1)}, -2^{-d}, L_1, L_{d-1}).
\]
Using Pitman's coagulation-fragmentation duality~\cite{zbMATH01496111}, we proceed to verify
\[
	W_{\mathrm{PD}}(2^{-(d-1)}, -2^{-d}, L_{1}, L_{d-1}) \,\, \eqdist \,\, W_{\mathrm{PD}}(1/2, -2^{-d},  L_{d-1}, L_{1}).
\]
This means we may let the Brownian graphon either govern the head or the components.
We close the section by studying a model of $2$-dimensional cographs with a component weight $q>0$. That is, we set head and component structures to cographs and draw a cograph of a given size $n$ with probability proportional to $q$ raised to the number of components. Our main result is that the phase diagram for graph limits admits a unique ``spike'' at the critical value $q=2 \log(2)-1$:
\[
	\mO^{\langle q \rangle}_n \convd \begin{cases}
		W_{\mathrm{Brownian}}, &q \neq 2 \log(2)-1 \\
		W_{\mathrm{PD}}(1/2, -1/4,  L_1, L_{1}), &q = 2 \log(2)-1.
	\end{cases}
\]

The second half of this manuscript is dedicated to developing a parallel theory for permutations. Limits of random permutations are called permutons~\cite{zbMATH06126921}. A permuton is a probability measure on the unit square $[0,1]$ with uniform marginal distributions. They serve to encode  asymptotic pattern densities in sequences of finite permutations. Section \ref{sec:permlim} recalls relevant background.

In Section \ref{sec:superperm} we introduce the Poisson-Dirichlet permuton $\mu_{\mathrm{PD}}(\alpha, \theta, L_{\cH}, L_{\cC})$ as a random permuton that is constructed from a $\mathrm{PD}(\alpha, \theta)$ process and laws $L_{\cH}, L_{\cC}$ of random permutons. Roughly speaking, we take the points of the Poisson-Dirichlet process as coefficients in a linear combination of push-forwards of independent samples of $L_{\cC}$ to square windows which are ordered using a single instance of $L_{\cH}$. The formal construction is given in Definition~\ref{def:pdperm}. We characterise induced finite permutations by $\mu_{\mathrm{PD}}(\alpha, \theta, L_{\cH}, L_{\cC})$ in terms of an $(\alpha,\theta)$ Chinese restaurant process and finite permutations induced by the random permutons with laws $L_{\cH}$ and $L_{\cC}$. Our main invariance principle, Theorem~\ref{te:main3} shows how Poisson-Dirichlet permutons arise as limits of random finite permutations in a heavy-tailed \emph{dilute} regime, and parallel to the phase diagram of graph limits we establish a  phase diagram of \emph{repellent regimes} in which permuton limits of random superpermutations are given by the limit of the head permutation in a \emph{dense} regime, the limit of a single large component in the \emph{condensation} regime, and mixture of the two laws in a \emph{mixture} regime.

In Section \ref{sec:supp} we proceed to demonstrate the universality of Poisson-Dirichlet permutons and their iterations by constructing an example related to the well-known class $\cP$ of separable permutations. Separable permutations are known as the class of all permutations avoiding the patterns 2413 and 3142~\cite{MR1620935}. It was shown by~\cite{zbMATH06919021} that random separable permutations admit a random permuton $\mu_{\mathrm{Brownian}}$ called the Brownian (separable) permuton as their limit. It may be constructed from a  Brownian excursion~\cite{zbMATH07359153}. We define $d$-dimensional superpermutations with a split, by setting $\cP_1 = \cP$ and recursively constructing the class $\cP_d$ of superpermutations with head structures in $\cP$ and component structures from the class of ordered pairs $(P, P')$ of superpermutations from $\cP_{d-1}$. The ordered pair $(P,P')$ is identified with a superpermutation in a canonical way via concatenation. We show that the number of $n$-sized superpermutations in $\cP_d$ grows asymptotically like
\[
	 (\sqrt{2}-1)\frac{2^{3/2-d-5/2^{d+1}}}{\Gamma(1-2^{-d})}n^{-1-2^{-d}} (3- 2 \sqrt{2})^{-n}.
\]
We let $\sigma_{n,d}$  denote a uniformly selected permutation of size $n$ in $\cP_d$, and set $L_1 = \mathfrak{L}(\mu_{\mathrm{Brownian}})$ to the law of the Brownian (separable) permuton. With $L_d$ denoting the law of the random permuton limit of  $\sigma_{n,d}$ for $d \ge 2$, we show in Theorem~\ref{te:main4} that
\[
		\sigma_{n,d} \convd \mu^{(d)} := \mu_{\mathrm{PD}}(2^{-(d-1)}, -2^{-d}, L_1, L_{d-1}).
\]
Using Pitman's coagulation-fragmentation duality, we derive an alternative representation
\[
	\mu_{\mathrm{PD}}(2^{-(d-1)}, -2^{-d}, L_1, L_{d-1}) \eqdist \mu_{\mathrm{PD}}(1/2, -2^{-d},  L_{d-1}, L_{1}).
\]
We also consider random $2$-dimensional $n$-sized superpermutations $\sigma^{\langle q \rangle}_n$ biased by a component weight $q>0$, with head and component structures in the class of separable permutations. That is, $\sigma^{\langle q \rangle}_n$ assumes such a superpermutation with probability proportional to $q$ raised to the number of components. In the phase diagram of permuton limits we observe a ``spike'' for the critical value $q= \sqrt{2}-1$
	\[
\sigma^{\langle q \rangle}_n \convd \begin{cases}
	\mu_{\mathrm{Brownian}}, &q \neq \sqrt{2}-1 \\
	\mu_{\mathrm{PD}}(1/2, -1/4, L_1, L_{1}), &q = \sqrt{2}-1.
\end{cases}
\]

\subsection*{Notation}
We let $\ndN$ denote the positive integers and $\ndN_0$ the non-negative integers.  For each $n \in \ndN_0$ we set $[n] = \{1, \ldots, n\}$. All   unspecified limits are taken as $n \to \infty$. Convergence in distribution and probability are denoted by $\convd$ and $\convp$. We let $O_p(1)$ denote an unspecified random variable $Y_n$ that is stochastically bounded as $n \to \infty$, and $o_p(1)$ an unspecified random variable that tends in probability to zero as $n$ tends to infinity.  The $n$th coefficient of a power series $a(z)$ is denoted by $[z^n]a(z)$. All graphs considered in this work are simple, that is, there are no multi-edges or loops.

\section{Two-parameter processes and partition structures}
\label{sec:twoparamproc}

\subsection{The two-parameter Poisson-Dirichlet process}

Recall that given $a,b>0$, the \emph{$\mathrm{Beta}(a,b)$ distribution} on $]0,1[$ has probability density function
\[
\frac{x^{a-1}(1-x)^{b-1}}{B(a,b)},
\qquad 0<x<1.
\]
Here 
\[
	B(a,b) = \frac{\Gamma(a) \Gamma(b)}{\Gamma(a+b)} = \int_{0}^1 t^{a-1}(1-t)^{b-1}\,\,\mathrm{d}t
\]
denotes Euler's Beta function.

Given $0<\alpha<1$ and $\theta>-\alpha$ the \emph{two-parameter Poisson-Dirichlet process $\mathrm{PD}(\alpha, \theta)$} introduced by Pitman and Yor~\cite{MR1434129}  is the point process corresponding to a sequence of random variables
\[
	(V_1, V_2, \ldots) \quad \text{with} \quad V_1 > V_2 > \ldots \quad \text{and} \quad \sum_{\ell=1}^\infty V_\ell = 1,
\]
that may be constructed  via a stick-breaking construction by starting with a sequence of independent Beta random variables
\[
\widetilde{Y}_{\ell} \eqdist \mathrm{Beta}\bigl(1-\alpha,\;\theta + \ell\alpha\bigr),
\quad \ell=1,2,\ldots\,,
\]
setting
\[
\widetilde{V}_{1} = \widetilde{Y}_{1},
\qquad
\widetilde V_{\ell}
=
\widetilde{Y}_{\ell}\,\prod_{i=1}^{\ell-1}(1-\widetilde{Y}_{i}),
\quad \ell\ge2,
\]
and letting $(V_\ell)_{\ell \ge1}$ denote the ranked values of $(\widetilde{V}_\ell)_{\ell \ge 1}$.

\subsection{The two-parameter Chinese restaurant process}
\label{sec:crp}

The two-parameter Poisson-Dirichlet process $(V_\ell)_{\ell \ge 1}$ induces a family of growing partitions: If we ``paint'' the unit interval by cutting it into pieces of random lengths $V_1, V_2, \ldots$ and dropping $k \ge 1$ independent $\mathrm{Unif}([0,1])$ distributed points, we obtain an induced  partition of $[k]$.  

Conditionally  on $(V_\ell)_{\ell \ge 1}$, the probability for obtaining a partition whose classes have sizes $k_1, \ldots, k_r$ in the order of their least element is given by
\[
	\sum_{\substack{i_1, \ldots, i_r \\ \text{distinct}}} \prod_{j=1}^r \binom{k - 1 - k_1 - \ldots - k_{j-1}}{k_j-1} V_{i_j}^{k_{j}}.
\]
This is because we have $\binom{k-1}{k_1-1}$ choices for the class containing $1$, $\binom{k-1-k_1}{k_2-1}$ choices for the class containing the least element of the remaining points, and so on. The product of binomial coefficients arising from this ordering simplifies by a quick calculation as follows:
	\begin{align*}
		s(k_1, \ldots, k_r) &:= \prod_{j=1}^r \binom{k - 1 - k_1 - \ldots - k_{j-1}}{k_j-1}  \\ 
		&=  \frac{k!}{ (k_1 + \ldots + k_r )(k_2 + \ldots + k_r) \cdot \ldots \cdot k_r } \prod_{j=1}^r \frac{1}{(k_j-1)!}.
	\end{align*}
The unconditional probability 
\[
p(k_1,\dots,k_r) = s(k_1, \ldots, k_r)  \sum_{\substack{i_1, \ldots, i_r \\ \text{distinct}}} \Exb{\prod_{j=1}^r V_{i_j}^{k_j}}
\] for obtaining $(k_1, \ldots, k_r)$ was determined by Pitman~\cite[Prop. 9]{MR1337249} as the probability for a seating arrangement of a \emph{two-parameter Chinese restaurant process with an $(\alpha, \theta)$ seating plan}
\begin{equation}
	\label{eq:EPPF}
	p(k_1,\dots,k_r)
	= s(k_1, \ldots, k_r)
	\frac{\GPoch{\theta}{\alpha}{r}}
	{\Poch{\theta}{k}}
	\,\prod_{j=1}^{r} \Poch{1-\alpha}{\,k_j - 1}\,.
 \end{equation}
Here, for $x,d \in \ndR$ and integers $m \ge 0$ we set
\[
\Poch{x}{m}
:=
x\,(x+1)\,(x+2)\,\cdots\,(x+m-1),
\quad \Poch{x}{0}=1,
\]
and more generally
\[
\GPoch{x}{d}{m}
:=
\prod_{i=0}^{m-1}(x + i\,d),
\quad \GPoch{x}{d}{0}=1.
\]
The partition of $[k]$ is hence distributed like the seating plan of the following process, where customers $1,2, \ldots, k$ walk into a restaurant. Customer $1$ sits at table~$1$. If customers $1, \ldots, i$ have been seated at $m$ tables with $n_1, \ldots, n_m \ge 1$ occupants, then customer $i+1$ sits at table $1 \le j \le m$ with probability $\frac{n_j- \alpha}{\theta +i}$, and at a new table $m+1$ with probability $\frac{\theta + m\alpha}{\theta + i}$.

\subsection{Pitman's coagulation-fragmentation duality}

Suppose that $0<\alpha_1, \alpha_2 <1$ and $\theta > - \alpha_1 \alpha_2$. We may fragment a two-parameter Chinese restaurant process with an $(\alpha_1 \alpha_2, \theta)$ seating plan and $k$ customers by subdividing each table according to an independent Chinese restaurant process with an $(\alpha_1, - \alpha_1 \alpha_2)$ seating plan. The result  $\mathrm{FRAG}_k$ is a partition of a partition.

Conversely, we may start with a $\mathrm{CRP}(\alpha_1, \theta)$ process with $k$ customers and then consider the resulting tables as customers in a conditionally independent $\mathrm{CRP}(\alpha_2, \theta / \alpha_1)$ process, therefore creating a partition of partitions $\mathrm{COAG}_k$.

Pitman's coagulation-fragmentation duality~\cite{zbMATH01496111} states that for all $k \ge 1$
\begin{align}
	\label{eq:duality}
	\mathrm{FRAG}_k \eqdist \mathrm{COAG}_k.
\end{align}
It follows from Equation~\eqref{eq:EPPF}, which allows us to express the probability to obtain a partition of a partition with given sizes for both models, and compute that they coincide.
See also~\cite[Thm. 5.23]{MR2245368},~\cite[Prop. 22]{MR1434129}.

\section{Gibbs partitions in the dilute regime}

\label{sec:gibbs}

\subsection{Gibbs partitions}
A \emph{partition} $P$ of an $n$-element set $S$ is a finite collection of disjoint subsets whose union equals $S$. Formally, $P$ is a multiset, as it may contain the empty set multiple times. The elements of $P$ are its \emph{components}. We let $\mathrm{Part}(S)$ denote the countably infinite collection of all partitions of $S$.

Given two sequences $\bm v=(v_0,v_1,\dots)$ and $\bm w=(w_0,w_1,\dots)$ of non-negative real numbers, we assign to each partition $P \in \mathrm{Part}(S)$ its weight
\[
u(P)=|P|!\,v_{|P|}\,\prod_{Q\in P}\bigl(|Q|!\,w_{|Q|}\bigr).
\]
The associated \emph{partition function} 
\[
	u_n = \frac{1}{n!} \sum_{P \in \mathrm{Part}(S)}  u(P)
\]
only depends on the cardinality $n = |S|$. In order for $u_n$ to be finite we require $w_0<1$.

The associated generating series $U(z) = \sum_{n \ge 0} u_n z^n$, $V(z) = \sum_{n \ge 0} v_n z^n$ and $W(z) = \sum_{n \ge 0} w_n z^n$ satisfy the relation
\[
	U(z) = V(W(z))
\]
as formal power series. We let $\rho_v,\rho_w,\rho_u$ denote  the radii of convergence of $V,W,U$. 

Whenever $u_n>0$, define the random \emph{Gibbs partition} $P_n$ of $[n]$ by
\[
\mathbb{P}(P_n=P)
=\frac{u(P)}{n! u_n},
\quad P\in\mathrm{Part}([n]).
\]
Let $N_n=|P_n|$ be its number of components, and write those component sizes in a uniformly at random selected order $(K_1,\dots,K_{N_n})$. 

\subsection{The dilute regime}

The typical number and sizes of components in the Gibbs partition model depend on the weight-sequences $\bm{v}, \bm{w}$. Various regimes have been identified~\cite{MR2245368,MR4780503}. In the \emph{dilute regime}, we assume \emph{criticality}
\begin{align}
	\label{eq:crit}
	\rho_w >0 \qquad \text{and} \qquad \rho_v= W(\rho_w)>0,
\end{align}
and the following regularity conditions. We assume that there exist constants $0<\alpha<1$, $-\infty < \beta < 1$ and slowly varying functions $L_v, L_w$ such that
\begin{align}
	\label{eq:condw}
	\sum_{k > n} w_k \rho_w^k \sim L_w(n) n^{-\alpha} \qquad \text{and} \qquad w_n \rho_w^n = O( L_w(n) n^{-\alpha-1} )
\end{align}
and
\begin{align}
	\label{eq:condv}
	 v_n \rho_v^n \sim L_v(n) n^{-\beta-1}.
\end{align}
By Karamata's theorem, a sufficient condition for~\eqref{eq:condw} is  \[w_n \rho_w^n \sim  \alpha L_w(n) n^{-\alpha-1}.\] 
The following result shows that in the dilute regime the maximal components scale at the order $n$ with fluctuations described by a two-parameter Poisson-Dirichlet process. This theorem generalizes~\cite[Cor. 3.15]{MR4780503} and~\cite[Prop. 3.1]{zbMATH08054723}, which were stated under stronger assumptions. 

\begin{theorem}
	\label{te:gibbsdilutegene}
	Suppose that conditions~\eqref{eq:crit},~\eqref{eq:condw} and~\eqref{eq:condv} are met. Define the following point process $	\Upsilon_n = \sum_{\substack{1 \le i \le N_n \\ K_{i} > 0}} \delta_{K_{i} / n}$ on $]0,1]$, with  $\delta$ referring to the Dirac measure. Then 
	\[
		\Upsilon_n \convd \mathrm{PD}(\alpha, -\alpha \beta)
	\]
	as $n \to \infty$.
\end{theorem}

Before we prove Theorem~\ref{te:gibbsdilutegene} we recall relevant facts on local large deviation and renewal theorems. Let $X \ge 0$ denote a random non-negative integer with probability generating function
\[
\Exb{z^X} = W(\rho_w z) / W(\rho_w).
\]
We let $X_1, X_2, \ldots$ denote independent copies of $X$ and set $S_n = \sum_{i=1}^n X_i$.  Equation~\eqref{eq:condw} ensures that 
\begin{align}
	\label{eq:Xfun}
\Prb{X > n} \sim W(\rho_w)^{-1} L_w(n) n^{-\alpha}.
\end{align}
The function $x \mapsto  x^\alpha / ( W(\rho_w)^{-1} L_w(x) )$ is asymptotically equivalent to a smooth, strictly increasing function $A(x)$.

Indeed, by~\cite[Prop. 1.2.5]{Mikosch1999regular} it  is asymptotically equivalent to a smooth positive function $A^*(x)$, and by Karamata's theorem \cite[Thm. 1.2.6]{Mikosch1999regular} the smooth function $A(x) := \alpha^{-1} \int_{1}^x A^*(t)/t \,\mathrm{d}t$ is asymptotically equivalent to $A^*(x)$ and clearly strictly increasing.

We let $a(x)$ denote the inverse function of $A(x)$ and set $a_n := a(n)$. This way, $a(x)$ varies regularly at infinity with exponent $1/\alpha$ and
\[
	n \Prb{X> a_n} \to 1
\]
as $n \to \infty$, and hence
\[
	S_n / a_n \convd X_\alpha
\]
for an $\alpha$-stable random variable $X_\alpha \ge 0$ given by its Laplace transform
\[
	\Exb{e^{-t X_\alpha}} = \exp\left(- \Gamma(1-\alpha)t^\alpha\right), \qquad \Re(t) \ge 0.
\]
Its density function $f_\alpha$ satisfies for all $x>0$
\begin{align*}
\Gamma(1-\alpha)^{1/\alpha}f(x\Gamma(1-\alpha)^{1/\alpha}) =  \frac{1}{\pi x} \sum_{k=1}^\infty \frac{\Gamma(k \alpha +1)}{k!} (-x^{-\alpha})^k \sin(- \alpha k \pi).
\end{align*}
See~\cite[Sec. 7, Thm. 4.1]{janson2022stabledistributions} for details.  Furthermore, by~\cite[Thm. 5.1, Ex. 5.5]{janson2022stabledistributions} it holds for all $s \in \ndC$ with $\Re(s) < \alpha$ that
\begin{align}
	\label{eq:xalphadilute}
	\Ex{X_{\alpha}^s} = \Gamma(1-\alpha)^{s/\alpha} \frac{\Gamma(1- s/\alpha)}{\Gamma(1-s)}.
\end{align}
By Gnedenko's local limit theorem~\cite[Thm. 4.2.1]{MR0322926}, it follows that uniformly  for $\ell = y A(n)$ with $y$ restricted to a compact subset of $]0,\infty[$ we have
\begin{align}
	\label{eq:gnedenko}
	\Prb{S_n = \ell} &\sim \frac{1}{a_\ell} f_\alpha(n / a_\ell) \sim \frac{1}{y^{1/\alpha} n} f_\alpha(y^{-1/\alpha}).
\end{align}
The following result by Doney describes the asymptotics of generalized Green's functions.
\begin{proposition}[{\cite[Thm. 3]{MR1440141}}]
	\label{pro:doney}
	Suppose that~\eqref{eq:condw} holds and let $b(\cdot)$ be regularly varying at $\infty$ with exponent $\gamma > -2$. Then
	\[
		\sum_{\ell=0}^\infty b_\ell \Prb{S_\ell = n} \sim b(A(n)) A(n) n^{-1}  \alpha \Ex{X_{\alpha}^{-\alpha(\gamma+1)}} 
	\]
\end{proposition}

By Kolchin's representation of Gibbs partitions~\cite[Thm. 1.2]{MR2245368}, component sizes are distributed like summands in a randomly stopped sum. That is,
\[
	\ (K_1, \ldots, K_{N_n})  \eqdist \left( (X_1, \ldots, X_N) \mid S_N= n\right)
\]
with $N \ge 0$ a random non-negative integer  independent from $(X_i)_{i \ge 1}$ that satisfies
\begin{align}
	\Pr{N = n}  = v_n \rho_v^n / V(\rho_v).
\end{align}
The partition function is up to an exponential tilting a special case of a generalized Green's function, hence by Proposition~\ref{pro:doney}, Equation~\eqref{eq:xalphadilute}, and Assumptions~\eqref{eq:crit},~\eqref{eq:condw}, and~\eqref{eq:condv}
\begin{align}
	\label{eq:partitionfucntion}
	u_n \rho_w^n &= \sum_{\ell \ge 0} v_\ell \rho_v^\ell  \Pr{S_\ell = n} \\
	&\sim L_v(A(n)) A(n)^{-\beta} n^{-1} \alpha \Ex{X_{\alpha}^{\alpha\beta}} \nonumber
\end{align}
and, by~\eqref{eq:gnedenko} and~\eqref{eq:condv} we have uniformly for $\ell = y A(n)$ with $y$ in a compact subset of $]0,\infty[$
\begin{align*}
	\Prb{N_n = \ell} &= \frac{v_\ell \rho_v^\ell}{ u_n \rho_w^n} \Pr{S_\ell = n} \\
	&\sim \frac{L_v(\ell)\ell^{-\beta-1}}{L_v(A(n)) A(n)^{-\beta} \alpha \Ex{X_{\alpha}^{\alpha\beta}}} y^{-1/\alpha}   f_\alpha(y^{-1/\alpha}) \\
	&\sim \frac{1}{A(n) \alpha \Ex{X_{\alpha}^{\alpha\beta}}} y^{-1/\alpha-\beta-1}   f_\alpha(y^{-1/\alpha}).
\end{align*}
This verifies
\begin{align}
	\label{eq:numcomb}
	N_n / A(n) \convd Z	
\end{align}
for a limiting random variable $Z>0$ with density 
\[
f_Z(y) = \frac{1}{ \alpha \Ex{X_{\alpha}^{\alpha\beta}}} y^{-1/\alpha-\beta-1}   f_\alpha(y^{-1/\alpha}).
\]
As a special case of~Proposition~\ref{pro:doney}, 
\[
	\Exb{(N_n / A(n))^r} \to \Exb{Z^r} 
\]
for all $r > \beta-1$. By~\eqref{eq:xalphadilute} and integration by substitution
\begin{align}
	\label{eq:momzr}
	\Ex{Z^r} = \frac{\Ex{(X_{\alpha})^{\alpha(\beta-r)}}}{\Ex{(X_{\alpha})^{\alpha\beta}}} = \Gamma(1-\alpha)^{-r} \frac{\Gamma(1-\beta+r) \Gamma(1- \alpha\beta)}{\Gamma(1 - \beta) \Gamma(1- \alpha(\beta-r))}.
\end{align}


\begin{proof}[Proof of Theorem~\ref{te:gibbsdilutegene}]	
	Let $K_{\langle 1 \rangle}, K_{\langle 2 \rangle}, \ldots$ denote the sizes of the components in the Gibbs partition model ordered according to their smallest element (with empty components placed at the end of this list). This is the size-biased ordering, that picks components with probability proportional to its size.

	Let $k$ denote an integer with $k/n$ restricted to fixed compact subset of $]0,1[$. Then, by Proposition~\ref{pro:doney}, and Assumptions~\eqref{eq:crit},~\eqref{eq:condw}, and~\eqref{eq:condv}, uniformly in $k$
	\begin{align*}
		\Prb{ K_{\langle 1 \rangle} = k } 
		&= \frac{k}{n} \Exb{\sum_{i=1}^{N} \one_{X_i = k}  \Big \vert S_{N} = n} \\
		&= \frac{k}{n} \frac{\Prb{X=k}}{\Prb{S_{N} = n}} \sum_{\ell \ge 1} \Prb{N = \ell }\ell  \Prb{S_{\ell-1} = n-k} \\
		&\sim \frac{k}{n} \Prb{X=k} \frac{L_v(A(n-k)) A(n-k)^{-\beta+1} (n-k)^{-1} \alpha \Ex{X_{\alpha}^{\alpha(\beta-1)}}}{L_v(A(n)) A(n)^{-\beta} n^{-1} \alpha \Ex{X_{\alpha}^{\alpha\beta}}} \\
		&\sim   A(n) \Prb{X=k}  \frac{k}{n}\left(1-\frac{k}{n}\right)^{(1- \beta)\alpha-1} \frac{ \Ex{X_{\alpha}^{\alpha(\beta-1)}}}{   \Ex{X_{\alpha}^{\alpha\beta}}}.
	\end{align*}
	By~\eqref{eq:xalphadilute} and standard properties of the Gamma function,
	\[
	\frac{ \Ex{X_{\alpha}^{\alpha(\beta-1)}}}{   \Ex{X_{\alpha}^{\alpha\beta}}} = \frac{1}{\alpha B(1-\alpha, \alpha(1-\beta))}.
	\]
	By~\eqref{eq:Xfun}, the measure $A(n) \Prb{n^{-1} X \in \cdot}$ converges vaguely towards a limiting measure on $]0, \infty[$ with density $\alpha x^{-\alpha -1}$, $x>0$. Consequently, $K_{\langle 1 \rangle} / n$ converges in distribution towards a limiting random variable with density
	\[
		\frac{1}{ B(1-\alpha, \alpha(1-\beta))}  x^{-\alpha} (1-x)^{(1-\beta)\alpha-1}, \qquad 0<x<1.
	\]
	In other words,
	\[
		K_{\langle 1 \rangle} /n \convd \mathrm{Beta}(1-\alpha, (1-\beta)\alpha).
	\]

	Let $k_1, \ldots, k_m$,  denote positive integers with $(k_i/n)_{1 \le i \le m}$ restricted to a fixed compact subset of the open simplex \[
	\Delta_{m} := \{(v_1, \ldots, v_m) \mid v_1, \ldots, v_m > 0, \sum_{i=1}^m v_i < 1 \}.
	\] Set $s = \sum_{j=1}^{m} k_j$. Then, by analogous arguments as for $m=1$,
	\begin{align*}
		&\Prb{ K_{\langle 1 \rangle} = k_1, \ldots,  K_{\langle m \rangle} = k_m} \\
		&=  \Exb{\sum_{i_1=1}^{N}  \frac{k_1 \one_{X_{i_1} = k_1}}{n} \sum_{i_2=1}^{N} \one_{i_1 \ne i_2} \frac{k_2 \one_{X_{i_2} = k_2}}{n - k_1} \cdots \Big \vert S_{N} = n} \\
		&= \frac{1}{\Prb{S_{N} = n}}\left( \prod_{i=1}^m \frac{k_i \Prb{X=k_i}}{n - \sum_{j=1}^{i-1} k_j} \right) \sum_{\ell \ge m} \Prb{N = \ell }\left( \prod_{i=0}^{m-1} (\ell -i) \right)  \Prb{S_{\ell-m} = n-s} \\
		&\sim \left( \prod_{i=1}^m \frac{k_i \Prb{X=k_i}}{n - \sum_{j=1}^{i-1} k_j} \right) \frac{L_v(A(n-s)) A(n-s)^{-\beta+m} (n-s)^{-1} \alpha \Ex{X_{\alpha}^{\alpha(\beta-m)}}}{L_v(A(n)) A(n)^{-\beta} n^{-1} \alpha \Ex{X_{\alpha}^{\alpha\beta}}} \\
		&\sim \left( \prod_{i=1}^m \frac{k_i \Prb{X=k_i}}{n - \sum_{j=1}^{i-1} k_j} \right) A(n)^m \left(1 - \frac{s}{n}\right)^{\alpha(m - \beta)-1}\frac{ \Ex{X_{\alpha}^{\alpha(\beta-m)}}}{   \Ex{X_{\alpha}^{\alpha\beta}}}.
	\end{align*}
	By~\eqref{eq:xalphadilute} and standard properties of the Gamma function,
	\[
	\frac{ \Ex{X_{\alpha}^{\alpha(\beta-m)}}}{   \Ex{X_{\alpha}^{\alpha\beta}}} = \alpha^{-m} \prod_{j=1}^m B(1-\alpha, \alpha(j-\beta))^{-1}.
	\]
	We know that the measure $A(n) \Prb{n^{-1} X \in \cdot}$ converges vaguely towards a limiting measure on $]0, \infty[$ with density $\alpha x^{-\alpha -1}$, hence $(K_{\langle i \rangle}/n)_{1 \le i \le m}$ converges in distribution towards a random vector 
	$
	(\hat{V}_i)_{1 \le i \le m} \in \Delta_{m}
	$
	with density function
	\[
		f(\bm{v}) = \left( \prod_{j=1}^m B(1-\alpha, \alpha(j-\beta))\right)^{-1} \left( \prod_{i=1}^m \frac{v_i^{-\alpha}}{1 - \sum_{j=1}^{i-1} v_j }  \right)  \left(1 - \sum_{j=1}^m v_j\right)^{\alpha(m-\beta)-1},
	\]
	$ \bm{v} = (v_i)_{1 \le i \le m} \in \Delta_{m}$.
	The bijection 
	\[
		\phi: \Delta_m \to ]0,1[^m, \quad \bm{v} = (v_i)_{1 \le i \le m} \mapsto (v_i / (1- v_1 - \ldots - v_{i-1})  )_{1\le i \le m}
	\]
	has inverse map
	\[
	 	\phi^{-1}: ]0,1[^m \to \Delta_m, \quad \bm{y} = (y_i)_{1 \le i \le m} \mapsto (y_i (1- y_{i-1}) \cdots (1-y_1)  )_{1\le i \le m}.
	\]
	It is easy to see that if $\bm{v}$ corresponds to $\bm{y}$ under this bijection, then
	\[
		\sum_{j=1}^{i-1} v_j = 1 - \prod_{j=1}^{i-1}(1 - y_j)
	\]
	for all $1 \le i \le m$.
	Here we use the common convention that the sum over an empty index set equals zero, and the product over an empty index set equals one.
	The Jacobian  of $\phi^{-1}$ is lower-triangular, hence its determinant is the product of the diagonal entries:
	\[
		\det D \phi^{-1} (\bm{y}) = \prod_{i=1}^m \prod_{j=1}^{i-1} (1-y_j).
	\]
	Setting $\hat{Y}_i = \phi(\hat{V}_1, \ldots, \hat{V}_m)$, it follows that $(\hat{Y}_i)_{1 \le i \le m} \in ]0,1[^m$ has density
	\begin{align*}
		&f(\phi^{-1}(\bm{y})) |\det D \phi^{-1} (\bm{y})| \\
		&= \left( \prod_{j=1}^m B(1-\alpha, \alpha(j-\beta))\right)^{-1} \left(\prod_{j=1}^m y_j(1 - y_{j-1})\cdots(1- y_1)\right)^{-\alpha} \left(\prod_{j=1}^m(1 - y_j)\right)^{\alpha(m-\beta)-1} \\
		&= \left( \prod_{j=1}^m B(1-\alpha, \alpha(j-\beta))\right)^{-1} \prod_{j=1}^m y_j^{-\alpha} (1-y_j)^{\alpha(j-\beta)-1}.
	\end{align*}
	Thus, $\hat{Y}_1,\ldots, \hat{Y}_m$ are independent, with $\hat{Y}_i$ following a $\mathrm{Beta}(1-\alpha, \alpha(i-\beta))$ distribution for all $1\le i \le m$. Consequently, $(\hat{V}_i)_{1 \le i \le m} \eqdist (\tilde{V}_i)_{1 \le i \le m}$ for the size-biased ordering $(\tilde{V}_i)_{i \ge 1}$ of the points of the  Poisson-Dirichlet process with the same parameter $\alpha$ and $\theta = -\alpha \beta$.
	
	Setting $K_{\langle i \rangle} = 0$ for $i >N_n$, we have $\sum_{i=1}^\infty K_{\langle i \rangle} /n = 1$ and $\sum_{i=1}^\infty \tilde{V}_i = 1$. Together with distributional convergence of finite dimensional marginals, it follows that
	\[
		( K_{\langle i \rangle} /n)_{i \ge 1} \convd (\tilde{V}_i)_{i \ge 1}.
	\]
	By the continuous mapping theorem, it follows that
		\[
	( K_{( i )} /n)_{i \ge 1} \convd (V_i)_{i \ge 1},
	\]
	and hence
		\[
	\Upsilon_n \convd \mathrm{PD}(\alpha, -\alpha \beta).
	\]
\end{proof}

\section{Background on graph limits}
\label{sec:graphons}

\subsection{Encoding graphs by homomorphism densities}

We recall background on graph limits following closely the exposition of~\cite{MR2463439}. All graphs considered in this paper are simple, that is, they have no loops or multi-edges. Given a finite graph $G$ and a sequence $v_1,\ldots,v_k$ of vertices of $G$, let $G(v_1,\ldots,v_k)$ be the graph with vertex set $[k]$ in which $i,j\in[k]$ are adjacent if and only if  $v_i$ and $v_j$ are adjacent in $G$. We allow the case $v_i=v_j$ for $i\ne j$; then $i$ and $j$ are not adjacent in $G(v_1,\ldots,v_k)$ since $G$ is simple.

Let $G[k]:=G(v_1',\ldots,v_k')$ with $v_1',\ldots,v_k'$ a sequence of independently and uniformly at random selected vertices of $G$. 
For any graph $H$ with vertex set $[k]$, the \emph{homomorphism density} is defined by
\begin{align}
	\label{eq:tind}
	t(H,G)=\Pr{H\subset G[k]}.
\end{align}
Here $H\subset G[k]$ means that every edge of $H$ is an edge of $G[k]$. Since both graphs share the same set $[k]$ of vertices this is well-defined.

The unlabelled graph underlying a labelled one is called its \emph{isomorphism type}. We let $\cU$ denote the countable set of all finite unlabelled graphs. If $U\in\cU$ is any graph with $k$ vertices, we choose an isomorphic graph $H$ with vertex set $[k]$ and set $t(U,G):=t(H,G)$. By symmetry this does not depend on the choice of $H$.

Define the map
\begin{align}
	\tau:\cU\to[0,1]^\cU,\qquad G\mapsto (t(U,G))_{U\in\cU}.
\end{align}
The map $\tau$ is not injective; for instance, the complete bipartite graph on $2n$ vertices  has the same image for all $n\ge1$. However, the  restriction of $\tau$ to unlabelled graphs with exactly $n$ vertices is injective for each $n\ge1$, see~\cite[Thm. 3.6]{MR214529}.

Equip $[0,1]^\cU$ with the product topology. Let $\cU^*:=\tau(\cU)$ and let $\overline{\cU^*}\subset[0,1]^\cU$ be its closure. Since $[0,1]^\cU$ is a countable product of compact Polish spaces, it is again  compact Polish; hence the closed subspace $\overline{\cU^*}$ is compact Polish.

For $\Gamma=(t_U)_{U\in\cU}\in[0,1]^\cU$ and $U\in\cU$ finite, set
\begin{align}
	t(U,\Gamma)=t_U.
\end{align}

\subsection{Encoding graphons by homomorphism densities}
\label{sec:encoding}

We follow the accounts in~\cite{MR3012035,MR2463439}. A \emph{graphon} (from “graph-function”) is a symmetric measurable function
\begin{align}
	W:[0,1]^2\to[0,1].
\end{align}
Let $\cW_{\mathrm{S}}$ be the set of graphons.

Given $W$, construct the random graph $G(\infty,W)$ on vertex set $\ndN$ by first sampling $X_1,X_2,\ldots\stackrel{\text{i.i.d.}}{\sim}\mathrm{Unif}[0,1]$, and then, for $i<j$, adding the edge $ij$ independently with probability $W(X_i,X_j)$. For $k\ge1$, let $G(k,W)$ be the induced subgraph on $[k]$.

If $H$ has vertex set $[k]$, define
\begin{align}
	t(H,W):=\Pr{H\subset G(k,W)}=\int_{[0,1]^k}\prod_{ij\in E(H)} W(x_i,x_j),\mathrm{d}x_1\ldots \mathrm{d}x_k.
\end{align}
Here $E(H)$ is the edge set of $H$, and $H\subset G(k,W)$ means that every edge of $H$ appears in $G(k,W)$. For $U\in\cU$ finite, pick any labelled copy $H$ with vertex set $[k]$ and set $t(U,W):=t(H,W)$; this is independent of the choice of $H$.

Define
\begin{align}
	\label{eq:premap}
	\cW_{\mathrm{S}}\to\overline{\cU^*},\qquad W\mapsto \Gamma_W:=(t(U,W))_{U\in\cU}.
\end{align}
One way to see this is well-defined is to note that the random graphs $(G(n,W))_{n\ge1}$ satisfy a.s.\ $\tau(G(n,W))\to\Gamma_W$ in $[0,1]^\cU$, hence $\Gamma_W\in\overline{\cU^*}$; see e.g.~\cite[Eq.~(6.1), Cor.~5.4, Thm.~5.5]{MR2463439}.

\begin{proposition}[{\cite[Thm. 7.1]{MR2463439}}]
	\label{pro:equiv}
	Let $W,W'\in\cW_{\mathrm{S}}$. The following are equivalent.
	\begin{enumerate}
		\item $\Gamma_W=\Gamma_{W'}$.
		\item $t(U,W)=t(U,W')$ for all $U\in\cU$.
		\item $G(\infty,W)\eqdist G(\infty,W')$ as random graphs on $\ndN$.
		\item $G(n,W)\eqdist G(n,W')$ as random graphs on $[n]$ for all $n\ge1$.
		\item There exist Lebesgue measure-preserving maps $\varphi,\varphi':[0,1]\to[0,1]$ such that $W(\varphi(x),\varphi(y))=W'(\varphi'(x),\varphi'(y))$ almost everywhere. That is, $W$ and $W'$ are \emph{weakly isomorphic}.
		\item The cut distance
		\[
		\delta_{\square}(W,W')=\inf_{\varphi,\varphi'}\ \sup_{A,B}\left|\int_{A\times B}\bigl(W(\varphi(x),\varphi(y))-W'(\varphi'(x),\varphi'(y))\bigr)\,\mathrm{d}x\,\mathrm{d}y\right|
		\]
		vanishes, where $\varphi,\varphi'$ range over all Lebesgue measure-preserving maps $[0,1]\to[0,1]$, and $A,B$ over all Lebesgue-measurable subsets of $[0,1]$.
	\end{enumerate}
\end{proposition}

Let $\widehat{\cW}_{\mathrm{S}}$ be the quotient metric space induced by the pre-metric $\delta_{\square}$. The map in~\eqref{eq:premap} induces a homeomorphism
\begin{align}
	\label{eq:homeo}
	\widehat{\cW}_{\mathrm{S}}\cong \overline{\cU^*}.
\end{align}
See~\cite[Rem. 6.1]{MR2463439}.

Let $G$ have vertices $v_1,\ldots,v_n$. Interpret $G$ as a graphon $W_G$ by setting $W_G=1$ on squares of the form $](a-1)/n,a/n[\ \times\ ](b-1)/n,b/n[$, $1\le a,b\le n$, whenever $v_a$ and $v_b$ are adjacent, and $W_G=0$ elsewhere on $[0,1]^2$. By item~(5) of Proposition~\ref{pro:equiv}, the ordering $v_1,\ldots,v_n$ is immaterial up to weak isomorphism. Moreover, this interpretation is compatible with homomorphism densities:
\begin{align}
	\label{eq:igual}
	t(U,G)=t(U,W_G)
\end{align}
for all $U\in\cU$.

After passing to the image in $\widehat{\cW}_{\mathrm{S}}$, the order $v_1,\ldots,v_n$  of vertices ceases to matter. Thus each unlabelled finite graph $U\in\cU$ has a canonical image $\widehat{W}_U\in\widehat{\cW}_{\mathrm{S}}$. Using~\eqref{eq:igual} and Proposition~\ref{pro:equiv}, $\tau(U)\in\overline{\cU^*}$ and $\widehat{W}_U\in\widehat{\cW}_{\mathrm{S}}$ correspond under the homeomorphism~\eqref{eq:homeo}.

\subsection{Graph limits}

A finite unlabelled graph may be viewed as an element of $\overline{\cU^*}$, or equivalently of $\widehat{\cW}_{\mathrm{S}}$. Hence a random finite graph is a random element of these spaces. Since $\widehat{\cW}_{\mathrm{S}}\cong\overline{\cU^*}$ are Polish, standard notions of convergence in distribution apply. The next criterion characterizes convergence in law.

\begin{lemma}[{\cite[Thm. 3.1]{MR2463439}}]
	\label{le:mega}
	For any sequence $(G_n)_{n\ge1}$ of random finite graphs whose number of vertices tends in probability to infinity, the following are equivalent.
	\begin{enumerate}
		\item The random graph $G_n$ (viewed as $\tau(G_n)\in\overline{\cU^*}$) converges in distribution to some random element $\Gamma\in\overline{\cU^*}$.
		\item For any family $(H_i)_{1\le i\le \ell}$ of finite graphs, the vector $(t(H_i,G_n))_{1\le i\le \ell}$ converges in distribution.
		\item For any finite $H\in\cU$, the density $t(H,G_n)$ converges in distribution.
		\item For any finite $H\in\cU$, the expectation $\Ex{t(H,G_n)}$ converges.
	\end{enumerate}
	The limits in 2., 3., and 4.\ are necessarily $(t(H_i,\Gamma))_{1\le i\le \ell}$, $t(H,\Gamma)$, and $\Ex{t(H,\Gamma)}$.
\end{lemma}

By the discussion above, an entirely analogous statement holds if we view $G_n$ as $\widehat{W}_{G_n}\in\widehat{\cW}_{\mathrm{S}}$ instead. In particular, any of the four conditions in Lemma~\ref{le:mega} is equivalent to $\widehat{W}_{G_n}\convd W$ for some random element $W\in\widehat{\cW}_{\mathrm{S}}$, obtained by applying the homeomorphism $\overline{\cU^*}\cong\widehat{\cW}_{\mathrm{S}}$ in~\eqref{eq:homeo} to $\Gamma$.

By \cite[Cor. 3.2]{MR2463439}, if random elements $\Gamma,\Gamma'\in\overline{\cU^*}$ satisfy $\Ex{t(H,\Gamma)}=\Ex{t(H,\Gamma')}$ for all $H\in\cU$, then
\begin{align}
	\label{eq:unique}
	\Gamma\eqdist\Gamma'.
\end{align}
Thus, for $(\mG_n)_{n \ge 1}$ as in Lemma~\ref{le:mega}, the family $(t_H)_{H\in\cU}:=(\lim_{n\to\infty}\Ex{t(H,\mG_n)})_{H\in\cU}$ uniquely determines the law of the limit.

\section{Convergence of random supergraphs}
\label{sec:supergraphs}

\subsection{Supergraphs}

A class $\cC$ of labelled graphs is a collection of  graphs that is closed under relabelling of vertices. That is, if a graph belongs to the class $\cC$ then so does any isomorphic copy of it. For any non-negative integer $n \ge 0$ we let $\cC_n \subset \cC$ denote the subset of all graphs with vertex set $[n]$. Suppose that
\[
	\omega_{\cC}: \cC \to \ndR_{\ge 0}
\]
is a function satisfying $\omega_{\cC}(C) = \omega_{\cC}(C')$ if $C, C' \in \cC$ are isomorphic. We say $\omega_{\cC}$ assigns to each graph of $\cC$ its \emph{weight}. We let 
\[
	|\cC_n|_\omega = \sum_{C \in \cC_n} \omega_{\cC}(C)
\]
denote the \emph{inventory} of size $n$.
 The exponential generating series of the class $\cC$ with weight function $\omega_{\cC}$ is defined as the  power series
\[
	\cC(z) = \sum_{n \ge 1} \frac{|\cC_n|_{\omega_{\cC}}}{n!} z^n.
\]
We let $\rho_\cC$ denote its radius of convergence.

Given classes $\cH$ and $\cC$ of labelled graphs we may form the class \[
\cG= \cH \circ \cC 
\]
of all corresponding (\emph{two-dimensional}) supergraphs obtained by taking a graph $H$ from the class $\cH$ and blowing up each vertex $v$ of $H$ by a graph $C_v$ from the class $\cC$. Vertices from distinct components $C_v$ and $C_w$ for $v \ne w$ are adjacent if and only if the vertices $v$ and $w$ are adjacent in $H$. We keep the information on the ``head'' structure $H$, the underlying partition of the final vertex set and the bijective correspondence of these classes to the vertices of $H$ as a rooting in order to ensure that the decomposition into components is always unique.

With $\cH$ and $\cC$ equipped with weight functions $\omega_{\cH}$, $\omega_\cC$, the $\omega_\cG$ weight of a supergraph from $\cG$ is defined  as the product of the weight of its head-structure and the weights of its components.

We let $\mG_n$ denote the  random graph from the set $\cG_n$ drawn with probability proportional to its weight. Likewise, we define  random graphs $\mH_n$ from the set $\cH_n$ and $\mC_n$ from the set $\cC_n$, each with probability proportional to their weights. Of course, we only consider integers $n$ for which the inventories are positive. Throughout, we set
\[
	c_n = \frac{|\cC_n|_{\omega_{\cC}}}{n!} \qquad \text{and} \qquad h_n = \frac{|\cH_n|_{\omega_{\cH}}}{n!}.
\]
Furthermore, we set $\cC(z) = \sum_{n \ge 1} c_n z^n$ and $\cH(z) = \sum_{n \ge 1} h_n z^n$.

We emphasize that the special case of uniform random graphs  corresponds to setting all weights equal to $1$.

The classes $\cH$ and $\cC$ may also be allowed to be classes of supergraphs, therefore giving rise to supergraphs with a higher \emph{dimension}.

\subsection{Poisson-Dirichlet graphons in a dilute regime}

The head structure and component structure of $\mG_n$ may interact with each other in different ways. We establish a regime where both the asymptotic shapes of $\mC_n$ and $\mH_n$ influence the limiting shape of $\mG_n$, yielding a novel limiting objects governed by a two-parameter Poisson-Dirichlet process.

In this subsection, we assume
\begin{align}
	\label{eq:graphcrit}
	\rho_\cC >0 \qquad \text{and} \qquad \rho_\cH= \cC(\rho_\cC)>0,
\end{align}
as well as the following regularity conditions. We assume that there exist constants 
\begin{align}
	\label{eq:abet}
	0<\alpha<1, \qquad -\infty < \beta < 1
\end{align} and slowly varying functions $\ell_\cH, \ell_\cC$ such that
\begin{align}
	\label{eq:graphcondw}
	\sum_{k > n} c_k  \rho_\cC^k \sim \ell_\cC(n) n^{-\alpha} \qquad \text{and} \qquad c_n \rho_\cC^n = O( \ell_\cC(n) n^{-\alpha-1} )
\end{align}
and
\begin{align}
	\label{eq:graphcondv}
	h_n \rho_\cH^n \sim \ell_\cH(n) n^{-\beta-1}.
\end{align}
Suppose that there exist random graphons $W_\cH$ and $W_\cC$ such that  
\begin{align}
	\label{eq:compconv}
	\mH_n \convd W_{\cH} \qquad \text{and} \qquad \mC_n \convd W_{\cC}
\end{align}
as graphons. See Lemma~\ref{le:mega} for this notion of convergence.

We define a novel random graphon $W_{\cG}$ whose distribution depends on $\alpha$, $\beta$ (subject to~\eqref{eq:abet}), and the laws $L_{\cH} = \mathfrak{L}(W_{\cH})$ and $L_{\cC} = \mathfrak{L}(W_{\cC})$. The definition involves the two-parameter Poisson-Dirichlet process with parameters $\alpha$ and
\begin{align}
	\theta := - \alpha \beta.
\end{align}
In order to stay consistent with parametrization of the Poisson-Dirichlet process, we use $\theta$ as parameter for the random graphon $W_{\cG}$ instead of $\beta$.
\begin{definition}[Poisson-Dirichlet graphon]
	\label{def:wg}
	We let \[
	W_{\cG} =  W_{\mathrm{PD}}(\alpha, \theta, L_{\cH}, L_{\cC})
	\]
	 denote the random graphon defined as follows.
	\begin{enumerate}
		\item Generate random points $(V_\ell)_{\ell \ge 1}$ according to $\mathrm{PD}(\alpha, \theta)$. Note that $V_\ell >0$ for all $\ell \ge 1$ and  $\sum_{\ell \ge 1} V_\ell = 1$. For all $k \ge 0$ set $s_k = \sum_{i=1}^{k} V_i$.
		\item Generate a family $(W_i)_{i \ge 1}$ of independent and identically distributed copies of the random graphon~$W_{\cC}$ with law $L_{\cC}$. 
		\item With $W_{\cH}$ the random graphon with law $L_{\cH}$, generate $G(\infty, W_{\cH})$ (independently from all other considered random variables) and for all $i,j \ge 1$ with $i \ne j$ set $q_{i,j}=1$ if $i$ and $j$ are adjacent in  $G(\infty, W_{\cH})$, and $q_{i,j}=0$ otherwise.
		\item   This allows us to define, for all $i, j \ge 1$, an open rectangle \[
		Q_{i,j} = ]s_{i-1}, s_{i}[ \times ]s_{j-1}, s_{j}[
		\] of length $V_i$ and height $V_j$. We set for all $0<x,y<1$
		\[
		W_{\cG}(( 1-x) s_{i-1} + x s_{i}, (1-y) s_{j-1} + y s_j) = \begin{cases}
			W_{i}(x,y), \quad \text{ if } i =j \\
			q_{i,j}, \quad \text{ if } i \ne j.
		\end{cases}
		\]
		\item On the measure zero set $([0,1] \times [0,1]) \setminus \bigcup_{i,j} Q_{i,j}$ set $W_{\cG}$ equal to zero.
	\end{enumerate}
\end{definition}

Given an integer $k \ge 1$, the random graph $G(k, W_{\cG})$ induced on $[k]$ (see Section~\ref{sec:encoding}) admits the following description involving a Chinese restaurant process with an $(\alpha, \theta)$ seating plan.

\begin{lemma}
	\label{le:finallanif}
	For each $k \ge 1$, the random graph $G(k, W_{\cG})$ is distributed like the outcome of the following procedure.
	\begin{enumerate}
		\item  Run an $(\alpha, \theta)$ Chinese restaurant process until step~$k$. This results in a number $r \ge 1$ of tables with $k_1, \ldots, k_r$ customers. Customers are represented by integers from $[k]$.
		\item Sample an independent copy $H$ of $G(r, W_{\cH})$. 
		\item For each $1 \le i \le r$ sample an independent copy $C_i$ of $G(k_i, W_{\cC})$. (That is, the $C_i$ do \emph{not} share an underlying instance of the graphon $W_{\cC}$. Conditionally on~$r$ they are fully independent.)
		\item  Customers $u,v \in [k]$ that are seated at differing tables $i, j \in [r]$, $i \ne j$ are declared adjacent in a graph on $[k]$ if and only if $i$ and $j$ are adjacent in the graph~$H$.
		\item For each $1 \le i \le r$, let $u_{i, 1} < \ldots < u_{i, k_i}$ denote the customers at table $i$, and set $u_{i, x}$ adjacent to $u_{i,y}$ for $x,y \in [k_i]$ if and only if $x$ and $y$ are adjacent in~$C_i$.
	\end{enumerate}
\end{lemma}
\begin{proof}
	We use the notation of Definition~\ref{def:wg}. In particular, $s_i = \sum_{\ell=1}^i V_\ell$ for the points $(V_\ell)_{\ell \ge 1}$ of a $\mathrm{PD}(\alpha, \theta)$ process. 
	Sample $k \ge 1$ uniform independent points $X_1, \ldots, X_k \in [0,1]$, such that $G(k, W_{\cG})$ is formed by letting conditionally on $W_{\cG}$ any vertices $a,b \in [k]$ be adjacent independently with probability $W_{\cG}(X_a, X_b)$. 
	
	Almost surely, all points fall into intervals of the form $]s_{i-1},s_i[$ for some $i \ge 1$. We may order the ``occupied'' intervals $I_1, \ldots, I_{\tilde{r}}$ according to the first time one of the points falls into them, with $\tilde{r}$ the total number of ``occupied'' intervals, and let  $\tilde{k}_1, \ldots, \tilde{k}_{\tilde{r}}$ denote the total number of points in each interval in this order.
	By the discussion in Section~\ref{sec:crp}, in particular Equation~\eqref{eq:EPPF},  we have
	\begin{align}
		\label{eq:ath}
		(\tilde{r}, \tilde{k}_1, \ldots, \tilde{k}_{\tilde{r}})  \eqdist (r, k_1, \ldots, k_r).
	\end{align}
	
	
	By construction of $W_{\cG}$, for $1 \le i,j \le \tilde{r}$, if $i \ne j$ then any point $a \in [k]$ with $X_a \in I_i$ is adjacent to any point $b \in [k]$ with $X_b \in I_j$ if and only if $i$ and $j$ are adjacent in $G(\infty, W_{\cH})$. If $i=j$, the points $X_{a_1}, \ldots, X_{a_t}$ belonging to $I_i$ are uniformly distributed (conditioned on belonging to $I_i$), hence the subgraph of $G(k, W_{\cG})$ formed by them is distributed like $G(t, W_{\cC})$ (after relabelling $s$ to $a_s$ for all $1 \le s \le t$), and conditionally independent from the subgraphs spanned from the points occupying other intervals. Together with Equation~\eqref{eq:ath}, this yields that $G(k, W_{\cG})$ is distributed exactly as stated, hence completing the proof.
\end{proof}

In the general setting considered here, the random graphon $W_{\cG}$ is the limit of the random graphs $\mG_n$ as the number of vertices $n$ tends to infinity.

\begin{theorem}
	\label{te:main1}
Suppose that Conditions \eqref{eq:graphcrit}--\eqref{eq:compconv} hold. Then
\[
	\mG_n \convd  W_{\mathrm{PD}}(\alpha, -\alpha \beta, L_{\cH}, L_{\cC})
\]
as random graphons.
\end{theorem}
\begin{proof}
	Recall that the information on the head structure of a graph in $\cG$ is kept.
	Hence the graph $\mG_n$ may be decomposed in a unique way into a head-structure with a random size $N_n$, and $N_n$ components with sizes $K_1, \ldots, K_{N_n}$ ordered by the smallest contained vertex label.

	The generating series of the involved graph classes are related by the composition operation
	\[
		\cG(z) = \cH(\cC(z)).
	\]
	Hence the number $N_n$ and sizes $K_1, \ldots, K_{N_n}$ of components are identically distributed as in the Gibbs partition model with weights $v_n = \frac{|\cH_n|}{n!}$ and $w_n = \frac{|\cC_n|}{n!}$ for all $n \ge 0$.  Conditions~\eqref{eq:graphcrit},~\eqref{eq:graphcondw} and~\eqref{eq:graphcondv} allow us hence to apply~Theorem~\ref{te:gibbsdilutegene}, yielding
	\begin{align}
		\label{eq:pdlimit}
		\sum_{\substack{1 \le i \le N_n \\ K_{i} > 0}} \delta_{K_{i} / n} \convd \mathrm{PD}(\alpha, \theta)
	\end{align}
	as point-processes for $\theta = -\alpha \beta$.
	
	Let $k \ge 1$ be given and choose vertices $v_1, \ldots, v_k$ of $\mG_n$ uniformly and independently at random.  Let $L_n$ denote the smallest index such that $v_1, \ldots, v_k$ belong to some of the $L_n$th largest components.  The points of a two-parameter Poisson-Dirichlet process sum up to one. Therefore, by~\eqref{eq:pdlimit},  $L_n$ is stochastically bounded.
	That is, $v_1, \ldots, v_k$ belong asymptotically to the macroscopic components whose joint distribution after rescaling by $n^{-1}$ is determined by~\eqref{eq:pdlimit}.
	
	The partition of the vertex set of $\mG_n$ into components induces a partition of $v_1, \ldots, v_k$ into components. Let $s \ge1$ denote the number of these components and $n_1, \ldots, n_s \ge 1$ with $\sum_{i=1}^s n_i = k$ the number of indices corresponding to each component, ordered by the smallest index.

	It follows by~\eqref{eq:EPPF} that for all $k_1, \ldots, k_r \ge 1$, $r \ge 1$ with $\sum_{i=1}^r k_i = k$ we have
	\begin{align}
		\lim_{n \to \infty} \Pr{ (n_1, \ldots, n_s) = (k_1, \ldots, k_r)} = p(k_1, \ldots, k_r),
	\end{align}
	the probability for a Chinese restaurant process with an $(\alpha, \theta)$ seating plan to yield $r$ tables with $k_1, \ldots, k_r$ customers at time $k$.

	Conditionally on its size, each component is drawn at random with probability proportional to its weight among all graphs from $\cC$ with the corresponding vertex set. Likewise, conditionally on its size the head-structure is drawn with probability proportional to its weight among all graphs from $\cH$ with a corresponding fixed vertex set of the same size.

	Using~\eqref{eq:compconv} and Lemma~\ref{le:finallanif}, it follows that as $n \to \infty$, 
	\[
	G(k, \mG_n) \convd G(k, W_{\mathrm{PD}}(\alpha, -\alpha \beta, L_{\cH}, L_{\cC})).
	\]
	Since this holds for all $k \ge 1$, it follows by Lemma~\ref{le:mega} that
	\[
		\mG_n \convd W_{\mathrm{PD}}(\alpha, -\alpha \beta, L_{\cH}, L_{\cC})
	\]
	as random graphons.
\end{proof}



\subsection{Repellent regimes}
\label{sec:graphrepel}

The preceding section describes a dilute regime where the asymptotic shapes of the head structure and component structures combine to a new limit object. There exist other regimes where instead the head structures or component structures dominate, or the limit is a mixture of the two limiting shapes.

\subsubsection{Dense regime}

We describe a regime where the head structure dominates. Suppose that $\rho_\cH>0$ and
\[
h_n = \ell_\cH(n) n^{-b}\rho_\cH^{-n} 
\]
for a slowly varying function $L_\cH$ and an exponent $b>1$. Furthermore suppose that one of the following  cases holds.
\begin{enumerate}[\qquad i)]
	\item 	We have $\rho_\cH = \cC(\rho_\cC)$ and 
	\[
	c_n = \ell_\cC(n) n^{- a} \rho_\cC^{-n}
	\]	
	for a slowly varying function  $\ell_\cC$ and an exponent $a>2$, such that  $1 < b <a$, or  $b=a$ and  $\ell_\cC(n) = o(\ell_\cH(n))$.
	\item We have $\rho_\cH < \cC(\rho_\cC)$ and $\gcd\{ n \ge 0 \mid c_n > 0\} = 1$.
\end{enumerate}
Furthermore, assume that
\[
	\mH_n \convd W_\cH
\]
for a random graphon $W_\cH$.

Under these assumptions, the graphon limit of $\mG_n$ is identical to the graphon limit of the head structure:

\begin{theorem}
	\label{te:dense}
	Under the preceding assumptions, we have $\mG_n \convd W_\cH$ as $n \to \infty$.
\end{theorem}
\begin{proof}
	As argued in the proof of Theorem~\ref{te:main1}, the number $N_n$ and sizes $K_1, \ldots, K_{N_n}$ of components in $\mG_n$ coincide with those of a Gibbs partition model with weight-sequences $(h_n)_{n \ge 0}$ and  $(c_n)_{n \ge 0}$.
	
	The assumptions allow us to apply~\cite[Thm. 3.1]{MR4780503}, yielding for $\kappa= \min(a-1,2)$ that
	\[
		\frac{N_n - n / \mu}{L(n)n^{1/\kappa}} \convd X_\kappa
	\]
	for some constant $\mu>0$, a slowly varying function $L(n)$ and a $\kappa$-stable random variable $X_\kappa$. Furthermore, by~\cite[Cor. 3.4]{MR4780503} 
	\[
		\frac{1}{L(n)n^{1/\kappa}} \max(K_1, \ldots, K_{N_n})  \convd \begin{cases}Y, &1<\kappa<2 \\ 0, &\kappa=2 \end{cases}
	\]
	for a Fr\'{e}chet-distributed random variable $Y \ge 0$.
	
	In particular,  the components scale at the order $o(n)$. Hence if we select vertices $v_1, \ldots, v_k$, $k \ge 1$ of $\mG_n$ independently at random, the probability that two or more fall into  the same component tends to zero as $n \to \infty$. Thus, in the likely event that they fall into distinct components,  any two vertices are adjacent if and only their components are adjacent in the macroscopic head structure.
	
	 By exchangeability it follows that conditional on the number $N_n$ of components, the indices of the up to $k$ components containing the selected vertices are uniformly distributed and independent. This proves
	 \[
	 	G(k, \mG_n) \convd G(k, W_{\cH})
	 \]
	as $n \to \infty$. Since this holds for arbitrarily large $k$, $\mG_n \convd W_{\cH}$ follows.	
\end{proof}

\subsubsection{Condensation regimes}

We describe a regime where a giant component dominates. We assume that there exists a random graphon  $W_\cC$ such that  
\begin{align*}
	 \mC_n \convd W_{\cC}
\end{align*}
as graphons. There are three subcases:

\vspace{1em}

\noindent \textbf{Convergent case:}
	The convergent regime assumes that $0<\cH'(\cC(\rho_\cC))< \infty$ and any of the following conditions hold.

	\begin{enumerate}[\qquad i)]
		\item We have $\rho_\cH = \cC(\rho_\cC)$ and \[
		h_n = \ell_\cH(n) n^{-b}\rho_\cH^{-n} \qquad \text{and} \qquad c_n = \ell_\cC(n) n^{- a} \rho_\cC^{-n}
		\]
		for slowly varying functions~$\ell_\cH$ and $\ell_\cC$, an exponent $b > 2$, and an exponent~$a$ such that either $1 < a < b$, or $a=b$ and $\ell_\cH(n) = o(\ell_\cC(n))$.
		\item We have $\infty \ge \rho_\cH > \cC(\rho_\cC)$ and \[
		\frac{c_n}{c_{n+1}} \to \rho_\cC > 0 \qquad \text{and} \qquad \frac{1}{c_n} \sum_{i+j = n} c_i c_j \to 2 \cC(\rho_\cC) < \infty\] as $n \to \infty$. (Note that the last two limits are automatically satisfied in case $c_n = \ell_\cC(n) n^{- a} \rho_\cC^{-n}$ with $a>1$ constant,  $\ell_\cC$ slowly varying, and $\rho_\cC > 0$.)
		\item We have $\sum_{k \ge 0} h_k \rho_{\cH}^k k^{1+a+\delta} < \infty$ and $c_n = \ell_\cC(n) n^{- a} \rho_\cC^{-n}$ for $a>1$, $\delta>0$ constants  and $\ell_\cC$ a slowly varying function.
		\item We have $c_n \sim c_\cC n^{- a} \rho_\cC^{-n}$ for constants $a>1$, $c_\cC > 0$, and additionally one of the following conditions hold.
		\subitem a) $\sum_{k \ge 0} h_k \rho_{\cH}^k k^{1+a} < \infty$.
		\subitem b) $1<a<2$.
		\subitem c) $a = 2$ and $\sum_{k \ge 0} h_k \rho_{\cH}^k k (\log k)^{2 + \delta} < \infty$ for some $\delta >0$.
		\subitem d) $2<a<3$ and $h_n \rho_{\cH}^n = o(n^{-a})$.
		\subitem e) $a=3$ and $h_n \rho_{\cH}^n = o(1 / (n^3 \log \log n))$.
	\end{enumerate}

These conditions originate from the study of randomly stopped sums~\cite{MR4038060} and composition schemes with subexponential densities~\cite{MR0348393,MR772907}.

\vspace{1em}

\noindent\textbf{Superexponential case:} In this regime we assume $\rho_\cH \in ]0,\infty]$ and $\rho_\cC= 0$. Furthermore, assume that for any $k \ge 2$ we have
\[
	\sum_{\substack{i_1 + \ldots i_k = n \\ 1 \le i_1, \ldots i_k < n - (k-1)}} c_{i_1} \cdots c_{i_k} = o(c_{n-(k-1)}).
\]
A sufficient condition for this equation is given~\cite[Lem. 6.17]{MR4132643}. We furthermore assume that $c_i>0$ for all large enough integers $i$.

\vspace{1em}

\noindent\textbf{Mesocondensation case: }  In this regime, we assume $\rho_\cH = \cC(\rho_\cC)$ and \[
h_n = \ell_\cH(n) n^{-b}\rho_\cH^{-n} \qquad \text{and} \qquad c_n = \ell_\cC(n) n^{- 1} \rho_\cC^{-n}
\]
for slowly varying functions~$\ell_\cH$ and $\ell_\cC$, and an exponent $b > -2$.

\vspace{1em}

\begin{theorem}
	\label{te:condensation}
	In any of the three preceding cases, we have $\mG_n \convd W_\cC$ as $n \to \infty$.
\end{theorem}
\begin{proof}
	As argued in the proof of Theorem~\ref{te:main1}, the number  and sizes  of components in $\mG_n$ coincide with those of a Gibbs partition model with weight-sequences $(h_n)_{n \ge 0}$ and  $(c_n)_{n \ge 0}$.
	
	In the convergent case, we may apply~\cite{MR4780503} (see also~\cite{MR3854044}), yielding that there exists a giant component with size $n - O_p(1)$, with the remainder admitting an almost surely finite distributional limit (hence the name convergent case).
	
	In the superexponential case, we may apply~\cite[Thm. 6.18]{MR4132643}, yielding that there exists a giant component and the total size of all other components is smaller than a fixed deterministic bound with high probability. 
	
	In the Mesocondensation case, we may apply~\cite[Thm. 1.1 and Thm. 1.2]{bosio2025gibbspartitionslatticepaths}, yielding that there exists a giant component with size $n - o_p(n)$ (and in contrast to the other two cases, the total number of components tends to infinity).
	
	It follows that in all three cases, if we draw a fixed number $k \ge 1$ of uniform random vertices of $\mG_n$, with high probability they all belong to the giant component. Since we assumed $\mC_n \convd W_{\cC}$, this proves
	\[
	G(k, \mG_n) \convd G(k, W_{\cC}).
	\]
	Consequently, $\mG_n \convd W_{\cC}$.
\end{proof}

\subsubsection{Mixture regime}
We describe a mixture regime, in which the limit of $\mG_n$ is a mixture of the limits of large head structures and large components.

Suppose that $\rho_\cH = \cC(\rho_\cC)$. Furthermore, suppose that
\[
h_n = \ell_\cH(n) n^{-a}\rho_\cH^{-n} \qquad \text{and} \qquad c_n = \ell_\cC(n) n^{- a} \rho_\cC^{-n}
\]	
for slowly varying functions~$\ell_\cH, \ell_\cC$ and a constant $a > 2$. Let 
\[
\mu = \cC(\rho_\cC)^{-1} \sum_{n \ge 1 } n c_n \rho_\cC^n.
\] 
Suppose that the limit
\[
p := \lim_{n \to \infty} \frac{\mu^{a-1}}{\cH'(\cC(\rho_\cC))}  \frac{\ell_\cH(n)}{\ell_\cC(n)} >0
\]
exists and is positive. Suppose that there exist random graphons $W_\cH$ and $W_\cC$ with
\begin{align*}
	\mH_n \convd W_{\cH} \qquad \text{and} \qquad \mC_n \convd W_{\cC}
\end{align*}
as graphons.

\begin{theorem}
	\label{te:mixture}
	Under the preceding assumptions, we have $0<p<1$ and \[
	\mG_n \convd B W_\cH + (1-B) W_{\cC}
	\] with $B$ denoting an independent $\mathrm{Bernoulli}(p)$ random variable.
\end{theorem}
\begin{proof}
	
	By~\cite[Thm. 3.12]{MR4780503}, the event $\cE_n= \{ N_n \ge n / (2 \mu)\}$ satisfies
	\begin{align*}
		\lim_{n \to \infty} \Pr{\cE_n} = p \in ]0,1[
	\end{align*}
	and the following conditional properties hold:
	\begin{enumerate}[\qquad a)]
		\item Conditioned on the event $\cE_n$, the Gibbs partition behaves as in the dense regime. Hence, by identical arguments as in the proof of~Theorem~\ref{te:dense} we have
		\begin{align*}
(\mG_n \mid \cE_n) \convd W_\cH.
		\end{align*}
		\item Conditioned on the complementary event $\cE_n^c$, the Gibbs partition behaves as in the convergent case of the condensation regime. Hence, by identical arguments as in the proof of~Theorem~\ref{te:condensation} we have
		\begin{align*}
			(\mG_n \mid \cE_n^c) \convd W_\cC.
		\end{align*}
	\end{enumerate}
	Thus,
	\[
	\mG_n \convd B W_\cH + (1-B) W_{\cC}
	\]
	for a $\mathrm{Bernoulli}(p)$ distributed random variable that is independent 
	of $W_\cH$ and $W_{\cC}$.
\end{proof}

\section{Finite dimensional marginals of the  Brownian tree}

\label{sec:browtre}

Let $\mT_n$ denote a critical Bienaym\'e--Galton--Watson tree conditioned on having $n$ leaves, whose offspring distribution has finite non-zero variance.

If we choose $k$ independent uniform random leaves from $\mT_n$, then together with the root vertex they span a subtree $R_k(\mT_n)$.  Call a vertex \emph{essential} if it is one of the $k$ selected leaves, a last common ancestor of any two of them, or the root of~$\mT_n$. The essential vertices form a tree $S_k(\mT_{n})$ with the genealogical structure inherited from $\mT_n$. We may recover $R_k(\mT_n)$ from $S_k(\mT_{n})$ by ordering the edges of $R_k(\mT_n)$ in some canonical way, and blowing up edges into paths with specified lengths $\bm{s}_n = (s_1, s_2, \ldots)$.

Let $\mS(k)$ denote a uniform random tree where each non-leaf has two children, except for the root which has one child. Such a tree has $2k-1$ edges.
It was shown by Aldous~\cite[Lem. 26]{MR1207226} that
\begin{align}
	\label{eq:rem1}
	\left(S_k(\mT_n), \frac{\sqrt{\Pr{\xi=0}\Va{\xi}}}{\sqrt{n}} \bm{s}_n\right) \convd (\mS(k), \bm{s})
\end{align}
with $\bm{s}$ an independent random $2k-1$ dimensional vector with positive coordinates whose density is given by
\[
3 \cdot 5 \cdots (2k-3) \left( \sum_{i=1}^{2k-1} x_i \right) \exp\left( -\left( \sum_{i=1}^{2k-1} x_i \right)^2 \right), \qquad x_1, \ldots, x_{2k-1} >0.
\]
In other words, $R_k(\mT_n)$ with edge lengths rescaled to $\frac{\sqrt{\Pr{\xi=0}\Va{\xi}}}{\sqrt{n}}$ converges in distribution to $\mR(k)$ with edge lengths set to the coordinates of $\bm{s}$. 

The result is a step in proving convergence of rescaled trees towards Aldous' continuum random tree, which is nowadays usually called the Brownian tree. To be precise, Aldous formulated this for $k$ random vertices in a Bienaym\'e--Galton--Watson tree conditioned on having $n$ vertices, but the proof is fully analogous for the setting of $k$ random leaves in a  tree  conditioned on having $n$ leaves. The only difference is that the additional factor $\sqrt{\Pr{\xi=0}}$ remains in~\eqref{eq:rem1}, which comes from the probability for a Bienaym\'e--Galton--Watson tree to have $n$ leaves. Furthermore,~\cite[Lem. 26]{MR1207226} is in fact a local limit theorem for $R_k(\mT_n)$, which entails that the vector of coordinate parities $\mathrm{par}(\bm{s}_n)$ with coordinates in $\{\mathrm{even},\mathrm{odd}\}$ satisfies
\begin{align}
	\label{eq:yooo}
	\left(S_k(\mT_n), \mathrm{par}(\bm{s}_n)\right) \convd (\mS(k), \bm{p}),
\end{align}
with $\bm{p}$ denoting a vector of $2k-1$ independent fair choices from $\{\mathrm{even},\mathrm{odd}\}$ that are also independent from $\mS(k)$. See~\cite[Thm. 6.6]{zbMATH07749496} and~\cite[Lem. 4.3]{MR4115736} for further extensions.

\section{Complement reducible supergraphs}

\label{sec:corsu}

\subsection{Background on complement reducible graphs}

We recall relevant background on the class $\cO$  of cographs (short for complement reducible graphs). It may be characterized as the smallest class of finite simple graphs that contains the graph with one vertex and is closed under taking graph complements and forming finite disjoint unions. 

Cographs may be encoded using cotrees~\cite{MR619603}. A \emph{generalized cotree} is a rooted unordered tree where the leaves carry labels and the internal vertices carry signs ($\oplus$ or $\ominus$). Internal nodes are required to have at least two children. The cograph corresponding to a generalized cotree is formed by taking the set of leaves (without any edges) and joining any two leaves by an edge if and only if their last common ancestor in the cotree has a plus sign. This representation is not necessarily unique, which is why for \emph{cotrees} (as opposed to generalized cotrees) we additionally require that the sign of any non-root vertex is the opposite from the sign of its parent. Thus, the sign of the root vertex determines all remaining signs. Any cograph has a unique representation as a cotree~\cite{MR619603}, if we adopt the convention that the tree consisting of a single vertex has no internal vertices.

Decomposing according to the degree of the root vertex yields that class $\cT$ of cotrees whose root-vertex has a plus sign satisfies the equation
\begin{align}
	\label{eq:cotsum}
	\cT(z) = z + \sum_{k \ge 2} \cT(z)^k / k!  = z + \cT(z) \hat{\phi}(\cT(z))
\end{align}
for $\hat{\phi}(z) = (\exp(z) - 1 -z ) / z$, which may be rewritten as
\begin{align}
	\label{eq:simtree}
	\cT(z) = z\phi(\cT(z))
\end{align}
with $\phi(z) = 1 / (1- \hat{\phi}(z)) =  z / (1 - \exp(z) + 2z)$. By symmetry,
\begin{align}
	\cO(z) = 2 \cT(z) - z.
\end{align}
Applying~\cite[Thm. 18.11, Rem. 7.5]{MR2908619} to Equation~\eqref{eq:simtree},  yields that $\cO(z)$ has radius of convergence 
\begin{align}
\rho_\cO = 2 \log(2) - 1
\end{align}
and
\begin{align}
	\label{eq:dobrasil}
	\cO(\rho_\cO) = 1
\end{align}
and
\begin{align}
	\label{eq:ipar}
	[z^n] \cO(z) \sim  \sqrt{\frac{2 \log(2) -1}{ \pi}} n^{-3/2} \rho_\cO^{-n}
\end{align}
as $n \to \infty$. This formula was first established in~\cite[Thm. 5]{MR2154567} using analytic methods instead.

From Equation~\eqref{eq:cotsum} it follows by a general principle~\cite[Lem 6.7]{MR4132643} (which in this case may also be easily verified directly), that a uniform cotree  with $n$ leaves is distributed like a Bienaym\'e--Galton--Watson tree conditioned on having $n$ leaves. The offspring distribution $\xi$ is given by 
\begin{align*}
	\Pr{\xi = 0} = 2 - 1 / \log(2), \quad \Pr{\xi=1} = 0, \quad \Pr{\xi = k} = \log^{k-1}(2) / k! \quad \text{for $k \ge 2$}.
\end{align*}
Clearly $\Ex{\xi}=1$ and  $0< \Va{\xi}<\infty$.
The sign of the root vertex is determined by a fair coin flip, and this determines the signs of all remaining inner vertices. 

The  uniform random cograph $\mO_n$ with $n$ (labelled) vertices is distributed like the cograph corresponding to this cotree. From Equation~\eqref{eq:yooo} it follows that if we select $k$ uniform independent leaves in the cotree, then the corresponding essential vertices together with the $\oplus$/$\ominus$ signs converge in distribution to the uniform random $k$-tree $\mS(k)$ with  signs on the internal vertices determined by independent fair coin-flips. Let $\mG(k)$ denote the corresponding graph. It follows that the random subgraph $\mO_n[k]$ (induced by the corresponding $k$ uniform independent random vertices of $\mO_n$) satisfies
\[
	\mO_n[k] \convd \mG(k).
\]
By Lemma~\ref{le:mega}, this entails that 
\begin{align}
	\label{eq:bgraphon}
	\mO_n \convd W_{\mathrm{Brownian}}
\end{align}
for some random graphon $W_{\mathrm{Brownian}}$. This was shown in~\cite[Thm. 1]{zbMATH07749496} by a similar short argument.  The random graphon $W_{\mathrm{Brownian}}$ is called the Brownian graphon, as it may be constructed from a  Brownian excursion~\cite[Sec. 4.2]{zbMATH07749496}. An independent proof via different methods is given in~\cite{zbMATH07749513}.

\subsection{Multidimensional super cographs with a prototype}

We set $\cO_1 = \cO$ to the class of cographs, and for each $d \ge 2$ we let $\cO_d$ denote the class of supergraphs with head structures from $\cO$ and component structures from the class of ordered pairs of a marked vertex and a graph from $\cO_{d-1}$ (with no edges between the marked vertex and its corresponding graph from $\cO_{d-1}$). Throughout this subsection we consider unweighted graphs (equivalently, each graph receives weight $1$). This way,
\[
	\cO_d(z) = \cO(z \cO_{d-1}(z)).
\]
The marked vertices in the components from a ``prototype'' with the adjacency relation inherited precisely from the head structure. We call graphs from $\cO_d$ $d$-dimensional complementary reducible graphs with a prototype.

\begin{lemma}
	\label{le:proasym}
	For each integer $d \ge 2$ we have 
	\[
	[z^n] \cO_d(z) \sim \frac{2^{-d}}{\Gamma(1- 2^{-d})} (2 \sqrt{\rho_\cO})^{2-2^{-(d-1)}} n^{-1-2^{-d}} \rho_\cO^{-n}
	\]
	with $\rho_\cO= 2\log(2)-1$ and $\cO_d(\rho_\cO) = 1$.
\end{lemma}
\begin{proof}
	For $d=1$ the formula agrees with~\eqref{eq:ipar}. Equation~\eqref{eq:dobrasil} yields $\cO_1(\rho_\cO)=1$. For $d \ge 2$, let us assume that $\cO_{d-1}(\rho_\cO)=1$ and that the asymptotic formula for $[z^{n}]\cO_{d-1}(z)$ holds. Then \[
	\cO_{d}(\rho_\cO) = \cO(\rho_\cO \cO_{d-1}(\rho_\cO))= \cO(\rho_\cO) = 1
	\] and we are in the situation of Equation~\eqref{eq:partitionfucntion} for $U(z) = \cO_d(z)$,  $V(z) = \cO(z)$ and $W(z) = z \cO_{d-1}(z)$. Specifically, $\rho_v=\rho_u=\rho_w = \rho_\cO$, $\beta=1/2$, $\alpha = 2^{-(d-1)}$, and
	\begin{align*}
		 L_v(x) &\sim \frac{\sqrt{\rho_\cO}}{ \sqrt{\pi}}, \\
		 L_w(x) &\sim \frac{2^{-(d-1)}}{\Gamma(1- 2^{-(d-1)})} (2 \sqrt{\rho_\cO})^{2-2^{-(d-2)}} (\rho_\cO/  \alpha), \\
		 A(x) &\sim x^\alpha / ( \rho_\cO^{-1} L_w(x) ).
	\end{align*}
	Hence, by~\eqref{eq:partitionfucntion} and~\eqref{eq:xalphadilute}
	\begin{align*}
		[z^n] \cO_d(z) &\sim L_v(A(n)) A(n)^{-\beta} n^{-1} \alpha \Ex{X_{\alpha}^{\alpha\beta}} \rho_\cO^{-n} \\
		&\sim \frac{\sqrt{\rho_\cO}}{\sqrt{\pi}} n^{-\alpha \beta-1} (L_w(n)\rho_\cO^{-1})^\beta \alpha \Ex{X_{\alpha}^{\alpha\beta}} \rho_\cO^{-n} \\
		&\sim C n^{-2^{-d}-1}\rho_\cO^{-n}
	\end{align*}
	for
	\begin{align*}
		C &\sim \frac{\sqrt{\rho_\cO}}{ \sqrt{\pi}}  (L_w(n)\rho_\cO^{-1})^\beta \alpha \Gamma(1-\alpha)^{\beta} \frac{\Gamma(1- \beta)}{\Gamma(1-\alpha\beta)} \\
		&\sim L_w(n)^{1/2} \alpha \frac{\Gamma(1- \alpha)^{1/2}}{\Gamma(1-\alpha/2)} \\
		&\sim 2^{1- 2^{-(d-1)}} (\sqrt{\rho_\cO})^{2-2^{-(d-1)}} 2^{-(d-1)} \frac{1}{\Gamma(1-2^{-d})} \\
		&=  (2 \sqrt{\rho_\cO})^{2-2^{-(d-1)}} 2^{-d} \frac{1}{\Gamma(1-2^{-d})} .
	\end{align*}
	This completes the proof.
\end{proof}
Let $\mO_{n,d}$ denote the uniform random $n$-vertex   $d$-dimensional complementary reducible graphs with a prototype. Let $L_{\mathrm{Brownian}} = \mathfrak{L}(W_{\mathrm{Brownian}})$ denote the law of the Brownian graphon.

\begin{theorem}
	\label{te:main2}
	Let $d \ge 2$ denote an integer. As $n\to \infty$,
	\[
	\mO_{n,d} \convd W^{(d)} := W_{\mathrm{PD}}(2^{-(d-1)}, -2^{-d}, L_{\mathrm{Brownian}}, L_{d-1})
	\]
	as random graphons. Here we set $L_1 = L_{\mathrm{Brownian}}$, and recursively 
	$
		L_{d-1} = \mathfrak{L}(W^{(d-1)} )
	$
	the law of the limit for $d-1$. 
\end{theorem}
\begin{proof}
	By Lemma~\ref{le:proasym}, we are in the dilute  regime for  $U(z) = \cO_d(z)$,  $V(z) = \cO(z)$ and $W(z) = z \cO_{d-1}(z)$ and hence $\beta=1/2$ and $\alpha= 2^{-(d-1)}$. Hence $\theta= -2^{-d}$.
	 For $d=2$, 
	Equation~\eqref{eq:bgraphon} yields graphon convergence of the head structure, and of the components. By Theorem~\ref{te:main1}, it follows that
	\[
		\mO_{n,2} \convd W^{(2)} := W_{\mathrm{PD}}(2^{-1}, -2^{-2}, L_{\mathrm{Brownian}}, L_{\mathrm{Brownian}}).
	\]
	By Induction on $d$ and Theorem~\ref{te:main1}, we obtain for all $d \ge 2$
	\[
		\mO_{n,d} \convd W^{(d)} := W_{\mathrm{PD}}(2^{-(d-1)}, -2^{-d}, L_{\mathrm{Brownian}}, L_{d-1}).
	\]
	This completes the proof.
\end{proof}

To keep the notation consistent, we set $W^{(1)} = W_{\mathrm{Brownian}}$. The limit object in Theorem~\ref{te:main2} admits an alternative parametrization:

\begin{proposition}
	\label{pro:char}
	As random graphons,
	\[
	W^{(d)}   \eqdist W_{\mathrm{PD}}(1/2, -2^{-d},  L_{d-1}, L_{\mathrm{Brownian}}).
	\]
\end{proposition}
\begin{proof}
	We proceed by induction on $d$. The base case $d=2$ is trivial by construction, as both sides of the distributional equality refer to the same object. As for the induction step, suppose that $d \ge 3$ and that
	\begin{align}
		\label{eq:ihyp}
				W^{(d-1)}   
				&\eqdist W_{\mathrm{PD}}(1/2, -2^{-(d-1)},  L_{d-2}, L_{\mathrm{Brownian}}). 
	\end{align}
	
	By Definition~\ref{def:wg},  $W^{(d)}$ is constructed from squares on the diagonal with side lengths following a $\mathrm{PD}(2^{-(d-1)}, -2^{-(d)})$ distribution. The squares on the diagonal are filled with independent copies of $W^{(d-1)}$, rescaled according to the side-lengths of the squares. The adjacency between these diagonal squares (more precisely, between the corresponding intervals) is determined by an independent copy of $G(\infty, W_{\mathrm{Brownian}})$.

	 By our induction hypothesis~\eqref{eq:ihyp}, $W^{(d-1)}$ may be realised via diagonal squares with side lengths following a $\mathrm{PD}(1/2, -2^{-(d-1)})$ distribution, filled with rescaled independent copies of $W_{\mathrm{Brownian}}$, and the adjacency between these squares is determined by an independent copy of $G(\infty, W^{(d-2)})$.
	
	Hence, in $W^{(d)}$ we have ``first level'' diagonal squares with side-length given by the ranked points $(X_i)_{i \ge 1}$ of a  $\mathrm{PD}(2^{-(d-1)}, -2^{-d})$ process, and for each $i \ge 1$ we cover the diagonal of the $i$th square by ``second level'' diagonal squares by sampling an independent copy $(X_{i,j})_{j \ge 1}$ of the ranked points of a $\mathrm{PD}(1/2, -2^{-(d-1)})$ process taking side-lengths equal to $(X_i X_{i,j})_{j \ge 1}$. 
	
	Applying Pitman's coagulation-fragmentation duality~\eqref{eq:duality} for  $\alpha_1 = 1/2$, $\alpha_2 = 2^{-(d-2)}$, $\theta= -2^{-d}$  yields that, up to re-ordering, $(X_i X_{i,j})_{i, j \ge 1}$ is distributed like the points of a $\mathrm{PD}(1/2, -2^{-d})$ process. Furthermore, the adjacency between these second level diagonals may be determined by grouping them according to a Chinese restaurant process with a $(2^{-(d-2)}, -2^{-(d-1)})$ seating plan and determining the adjacency between members of the same table by an independent copy of $G(\infty, W^{(d-2)})$  and the adjacency between tables by a single independent copy of $G(\infty, W_{\mathrm{Brownian}})$. By Lemma~\ref{le:finallanif}, this is precisely the manner of determining adjacency between points in $G(\infty, W^{(d-1)})$, therefore completing the proof.
\end{proof}

 \subsection{Weighted super-cographs}

Given a fixed constant $q>0$ we may consider two-dimensional super-cographs where each component receives weight $q$. That is, we consider the class $\cO^{\langle q \rangle} = \cH \circ \cC$ where $\cH = \cO$ with the trivial weighting that assigns weight $1$ to each graph, and $\cC = \cO$ with the weighting that assigns weight $q$ to each graph. The weight of a graph from $\cO^{\langle q \rangle}$ is hence $q$ raised to  the number of components, or equivalently, the size of the head-structure. Thus
\[
	\cO^{\langle q \rangle}(z) = \cO(q\cO(z)).
\]
We let $\mO^{\langle q \rangle}_n$ be drawn from $\cO^{\langle q \rangle}_n$ with probability proportional to its weight.

\begin{corollary}
	As $n \to \infty$,
	\[
		\mO^{\langle q \rangle}_n \convd \begin{cases}
			W_{\mathrm{Brownian}}, &q \neq 2 \log(2)-1 \\
			W^{(2)}, &q = 2 \log(2)-1.
		\end{cases}
	\]
	as random graphons.
\end{corollary}
\begin{proof}
	By Lemma~\ref{le:proasym} applied to $d=1$, we are in the dilute regime for $q = 2 \log(2)-1$, and in this case
	\[
		\mO^{\langle q \rangle}_n \convd W^{(2)} = W_{\mathrm{PD}}(1/2, -1/4, L_{\mathrm{Brownian}}, L_{\mathrm{Brownian}})
	\]
	by Theorem~\ref{te:main1}. For $0<q<2 \log(2)-1$ we are the  condensation regime (specifically the convergent case) and 
	\[
		\mO^{\langle q \rangle}_n \convd W_{\mathrm{Brownian}}
	\]
	by Theorem~\ref{te:condensation}. For $q> 2 \log(2)-1$ we are in the dense regime and 
	\[
	\mO^{\langle q \rangle}_n \convd W_{\mathrm{Brownian}}
	\]
	by Theorem~\ref{te:dense}.
\end{proof}

\section{Background on permutation limits}
\label{sec:permlim}

\subsection{Pattern densities of permutations}
We recall background on permutation limits following~\cite{zbMATH06126921}. For each $n \ge 1$ we let $\cS_n$ denote the symmetric group of degree $n$, and set $\cS= \bigcup_{n \ge 1} \cS_n$.
Given a permutation $\sigma \in \cS_n$ and a subset $I \subset [n]$ with $k := |I|$ we let $\varphi_I: I \to [k]$ denote the unique order-preserving map from $I$ to $[k]$. This allows us to  define the \emph{pattern} 
\[
\pat_I(\sigma) = \varphi_{\sigma(I)} \sigma \varphi_I^{-1} \in \cS_k.
\]
We let  $I_{n,k} \subset [n]$ denote a uniformly selected $k$-element subset of $[n]$. For each $\nu \in \cS$, we define the \emph{pattern density}
\[
	\mathrm{occ}(\nu, \sigma) = \Prb{ \pat_{I_{n,k}}(\sigma) = \nu}.
\]

\subsection{Permutons}

A \emph{permuton} $\mu$ is a Borel probability measure on the unit square $[0,1]^2$ with uniform marginals. That is, for all $0 \le a \le b\le 1$
\[
\mu( [0,1] \times [a,b] ) = \mu( [a,b] \times [0,1] ) = b-a.
\]

Let $(x_i, y_i)_{i \ge 1}$ denote i.i.d. samples of $\mu$. Then for each $n \ge 1$, almost surely $x_1, \ldots, x_n$ are pairwise distinct, and $y_1, \ldots, y_n$ are pairwise distinct. Let $\varphi: \{y_1, \ldots, y_n\} \to [n]$ denote the unique order preserving bijection and let $(1), \ldots, (n) \in [n]$ such that $x_{(1)} < \ldots < x_{(n)}$. We let $\mathrm{Perm}(n,\mu) \in \cS_n$ denote the permutation that maps $i \in [n]$ to $\varphi(y_{(i)}) \in [n]$.

A family $(\nu_n)_{n \ge 1}$ of random permutations $\nu_n \in \cS_n$ is called \emph{consistent}, if for all $1 \le k \le n$  we have $\pat_{I_{n,k}}( \nu_n) \eqdist \nu_k$. 

\begin{proposition}[{\cite[Prop. 2.9]{zbMATH07286839}}]
	The following statements hold:
	\begin{enumerate}
		\item 
	If $\mu$ is a random permuton, then $(\mathrm{Perm}(n,\mu))_{n \ge 1}$ is consistent. 
	\item If $(\nu_n)_{n \ge 1}$ is consistent, then there exists a random permuton $\mu$ with $\nu_n \eqdist \mathrm{Perm}(n,\mu)$ for all $n \ge 1$. $\mu$ is unique in the sense that any random permuton $\mu'$ with the same property must satisfy $\mu \eqdist \mu'$.
	\end{enumerate}
\end{proposition}

\subsection{Permutation limits}
\label{sec:plim}

We may view a permutation $\sigma \in \cS_n$ as  a permuton 
\[
\mu_\sigma = n \sum_{i=1}^n \mathrm{Lebesgue}( ([(i-1)/n, i/n]\times[(\sigma(i)-1)/n,\sigma(i)/n]) \cap \cdot)
\]
with $\mathrm{Lebesgue}(B \cap \cdot)$ denoting the Lebesgue measure restricted to a measurable subset $B \subset \ndR^2$. A random permutation $\sigma_n$ converges weakly to a random permuton $\mu$ as $n \to \infty$ if 
\[
	\mu_{\sigma_n} \convd \mu
\]
as random probability measures. There are different characterisations for this form of convergence: 
\begin{proposition}[{\cite[Thm. 2.5]{zbMATH07286839}}]
	\label{pro:doit}
For each $n \ge 1$ let $\sigma_n \in \cS_n$ denote a random permutation. The following statements are equivalent.
\begin{enumerate}[i)]
	\item There exists a permuton $\mu$ such that $\mu_{\sigma_n} \convd \mu$.
	\item For any integer $k \ge 1$ the pattern $\pat_{{I}_{n,k}}(\sigma_n)$ admits a distributional limit as random element of the finite set $\cS_k$.
	\item For each $n \ge 1$ and $\nu \in \cS$,  $\Exb{\mathrm{occ}(\nu, \sigma_n)}$ converges.
	\item $(\mathrm{occ}(\nu, \sigma_n))_{\nu \in \cS}$ converges in distribution with respect to the product topology.
\end{enumerate}
\end{proposition}

\section{Convergence of superpermutations}

\label{sec:superperm}

\subsection{Superpermutations}

We may represent any permutation $\theta \in \cS_\ell$ by a diagram of $\ell$ unit squares $[i-1, i] \times [\theta(i)-1, \theta(i)]$, $i=1, \ldots, \ell$. Let $Q$ denote one of these squares. We may perform a \emph{blow-up} of $Q$ by replacing it with the diagram of another permutation $\nu \in \cS_k$, which shifts the squares above $Q$ further above and the squares to the right of $Q$ further to the right to make room. This results in a diagram that represents a permutation of size $(\ell - 1) + k$.

We may also perform this blow-up for all squares at the same time. Given permutations $\nu_i \in \cS_{k_i}$, $1 \le i \le \ell$, $k_1, \ldots, k_\ell \ge 1$,  we may form a new permutation $\theta[\nu_1, \ldots, \nu_\ell] \in \cS_n$ for $n=k_1 +\ldots + k_\ell$ via the \emph{substitution operation}, by blowing up the $i$th square of $\theta$ with the diagram of $\nu_i$ for all $1 \le i \le \ell$. 

Analogously as for supergraphs, it will be convenient to refer to $\theta$ as the \emph{head structure} and $\nu_1, \ldots, \nu_\ell$ as the \emph{components} of the substitution.

It is easy to see that with $\bullet \in \cS_1$ the unique permutation of size $1$, we may write $\theta = \theta[\bullet, \ldots, \bullet] = \bullet[\theta]$. Thus, any permutation of size at least $2$ may be represented in at least two ways as a substitution. There may be even more ways, for example for the identity permutation of size at least $3$.

We define a \emph{(two-dimensional) superpermutation} as a permutation obtained via the substitution operation which keeps the information on the head structure and components. This way, we may refer to the head and components of a superpermutation. 

If we restrict the head-structure to belong to some subset $\cH \subset \cS$ and the components to some subset $\cC \subset \cS$, we thus form the class $\cG= \cH \circ \cC$ of superpermutations with these constraints. We deliberately use the same notation as for supergraphs to highlight the analogy.

Given \emph{weight-functions} $\omega_\cH: \cH \to \ndR_{\ge 0}$ and $\omega_\cC: \cC \to \ndR_{\ge 0}$, we may define the $\omega_\cG$ weight of a superpermutation from $\cG$ to be the product of the weights of its head structure and components. We let $\sigma^\cG_n$, $\sigma^\cH_n$ and $\sigma^\cC_n$ denote permutations of size $n$ drawn from the respective classes $\cG$, $\cH$ and $\cC$ with probability proportional to their $\omega_\cG$, $\omega_\cH$ and $\omega_\cC$ weight. Of course, this is only well-defined when at least one permutation of that size with positive weight exists.

We set
\[
	c_n = \sum_{\nu \in \cS_n \cap \cC} \omega_{\cC}(\nu) \qquad \text{and} \qquad h_n = \sum_{\nu \in \cS_n \cap \cH} \omega_{\cH}(\nu).
\]
Furthermore, we set $\cH(z) = \sum_{n \ge 1} h_n z^n$ and $\cC(z) = \sum_{n \ge 1} c_n z^n$.

The classes $\cH$ and $\cC$ may also be allowed to be classes of superpermutations, therefore giving rise to superpermutations with a higher \emph{dimension}. Whenever we don't explicitly specify a weight-function we refer to the weights that assign weight $1$ to each element of the class.

\subsection{Poisson-Dirichlet permutons in a dilute regime}

We introduce a class of random permutons governed by a two-parameter Poisson-Dirichlet process and establish a dilute regime in which these occur as universal limiting objects of random superpermutations.

Throughout this subsection we make the following assumptions:
\begin{align}
	\label{eq:pgraphcrit}
	\rho_\cC >0 \qquad \text{and} \qquad \rho_\cH= \cC(\rho_\cC)>0.
\end{align}
We assume that there exist constants 
\begin{align}
	\label{eq:pabet}
	0<\alpha<1, \qquad -\infty < \beta < 1
\end{align} and slowly varying functions $\ell_\cH, \ell_\cC$ such that
\begin{align}
	\label{eq:pgraphcondw}
	\sum_{k > n} c_k  \rho_\cC^k \sim \ell_\cC(n) n^{-\alpha} \qquad \text{and} \qquad c_n \rho_\cC^n = O( \ell_\cC(n) n^{-\alpha-1} )
\end{align}
and
\begin{align}
	\label{eq:pgraphcondv}
	h_n \rho_\cH^n \sim \ell_\cH(n) n^{-\beta-1}.
\end{align}
Suppose that there exist random permutons $\mu_\cH$ and $\mu_\cC$ such that  
\begin{align}
	\label{eq:pcompconv}
	\sigma_n^\cH \convd \mu_{\cH} \qquad \text{and} \qquad \sigma_n^\cC \convd \mu_{\cC}
\end{align}
as random permutons. See Section~\ref{sec:plim} for this notion of convergence.

We define a  random permuton $\mu_{\cG}$ whose distribution depends on $\alpha$, $\beta$ (subject to~\eqref{eq:pabet}), and the laws $L_{\cH} = \mathfrak{L}(\mu_{\cH})$ and $L_{\cC} = \mathfrak{L}(\mu_{\cC})$. The construction makes use of the two-parameter Poisson-Dirichlet process with parameters $\alpha$ and
\begin{align}
	\theta := - \alpha \beta.
\end{align}
In order to stay consistent with parametrization of the Poisson-Dirichlet process, we use $\theta$ as parameter for the random permuton $\mu_{\cG}$ instead of $\beta$.
\begin{definition}[Poisson-Dirichlet permuton]
	\label{def:pdperm}
		We let
		\[
			\mu_{\cG} = \mu_{\mathrm{PD}}(\alpha, \theta, L_{\cH},  L_ {\cC} )
		\]	
		denote the random permuton defined as follows.
\begin{enumerate}
	\item Let $(V_{\ell})_{\ell \ge 1}$ denote the points of a $\mathrm{PD}(\alpha, \theta)$ process.
	\item Let $(\mu_i)_{i \ge 1}$ denote independent copies of the random permuton $\mu_{\cC}$.
	\item Generate independent samples $(X_i, Y_i)_{i \ge 1}$ of  $\mu_{\cH}$. (This  means, we generate an instance of the random permuton $\mu_{\cH}$ and use the same instance for generating the $(X_i, Y_i)$, $i=1, \ldots$.) Since $\mu_{\cH}$ has uniform marginals, we have almost surely $X_i \ne X_j$ and $Y_i \ne Y_j$ for all $i \ne j$.
	\item Define the random c\`adl\`ag functions
	\[
		F(x) = \sum_{i \ge 1} V_i \one_{X_i \le x} \qquad \text{and} \qquad G(y) = \sum_{i \ge 1} V_i \one_{Y_i \le y}
	\]
	with jumps $V_i = F(X_i) - F(X_i-) = G(Y_i) - G(Y_i -)$ at $X_i$ and $Y_i$. Define
	\[
		A_i: [0,1]^2 \to [0,1]^2, \quad (x,y) \mapsto (F(X_i-) + V_ix, G(Y_i-) + V_iy).
	\]
	This way, $A_i([0,1]^2) = [F(X_i-), F(X_i)] \times [G(Y_i-), G(Y_i)]$.
	Define
	\[
		\mu_{\cG} = \sum_{i \ge 1} V_i (\mu_i \circ A_i^{-1}),
	\]
	with $\mu_i \circ A_i^{-1}: B \mapsto \mu_i(A_i^{-1}(B))$ denoting the pushforward of $\mu_i$ along $A_i$, and $V_i (\mu_i \circ A_i^{-1})$ denoting the product of $V_i$ and $\mu_i \circ A_i^{-1}$.
\end{enumerate}
\end{definition}
\noindent To see that  $\mu_\cG$ is a random permuton, note first that $\mu_\cG$ is a linear combination of (random) finite measures on $[0,1]^2$ and that
\[
\mu_{\cG}([0,1]^2) = \sum_{i \ge 1} V_i \mu_i([0,1]^2) = \sum_{i \ge 1} V_i = 1.
\]
This verifies that $\mu_{\cG}$ is a random probability measure on $[0,1]^2$. To see that it has uniform marginals, write $A_i = (A_{i,1}, A_{i,2})$. Let $\lambda(\cdot)$ denote the one-dimensional Lebesgue measure. Since the $\mu_i$ have uniform marginals, we have for all $0 \le a < b \le 1$, $I=[a,b]$  that
\begin{align*}
	\mu_{\cG}([0,1] \times I) &= \sum_{i \ge 1} V_i \mu_i([0,1] \times A_{i,2}^{-1}(I)  ) \\
	&= \sum_{i \ge 1} V_i \lambda(A_{i,2}^{-1}(I)) \\
	&= \sum_{i \ge 1}  \lambda(I \cap [G(Y_i-), G(Y_i)]) \\
	&= \lambda(I),
\end{align*}
and analogously
\begin{align*}
	\mu_{\cG}(I \times [0,1]) = \lambda(I).
\end{align*}

\begin{lemma}
	\label{le:finpdperm}
	For each $k \ge 1$, the random permutation $\mathrm{Perm}(k, \mu_{\cG})$ is distributed like the outcome of the following procedure.
	\begin{enumerate}
		\item  Run an $(\alpha, \theta)$ Chinese restaurant process until step~$k$. This results in a number $r \ge 1$ of tables with $k_1, \ldots, k_r$ customers. Customers are represented by integers from $[k]$.
		\item Sample an independent copy $\nu$ of $\mathrm{Perm}(r, \mu_{\cH})$. 
		\item For each $1 \le i \le r$ sample an independent copy $\nu_i$ of $\mathrm{Perm}(k_i, \mu_{\cC})$. (That is, let $\mu_i$ denote an independent copy of $\mu_\cC$ and set $\nu_i$ to be the permutation induced from $k_i$ samples of $\mu_i$.)
		\item Return the substitution $\nu[\nu_1, \ldots, \nu_r]$.
	\end{enumerate}
\end{lemma}
\begin{proof}
	Sampling from $\mu_\cG$ means choosing an index $i \ge 1$ with probability $V_i$, then sampling $(u,v)$ according to $\mu_i$, and returning $A_i(u,v)$.
	
	This way, if we sample $k \ge 1$ points according to $\mu_\cG$ we hit $r \ge 1$ of pairwise distinct indices $i_1, \ldots, i_r \ge 1$, $r \ge 1$ (in the order of the first time hit an index), such that $i_j$ got chosen a number $k_j \ge 1$ times, $\sum_{j=1}^r k_j = k$.
	
	By the argument of Section~\ref{sec:crp}, in particular Equation~\eqref{eq:EPPF}, it follows that the number $r \ge 1$ of indices and their multiplicities $k_1, \ldots, k_r$  are distributed like the number of tables and their occupants in an $(\alpha, \theta)$ Chinese restaurant process  with $k$ customers.
	
	Note that $A_i(]0,1[^2) = ]F(X_i-), F(X_i)[ \times ]G(Y_i-), G(Y_i)[$, with $F$ and $G$ being monotonically increasing. Hence, for $i, i' \ge 1$, we have that the $x$-range of $A_i$ lies fully to the left of the $x$-range of $A_{i'}$ or fully to its right, depending on whether $X_i < X_{i'}$ or $X_i > X_{i'}$. Likewise, the $y$-range of $A_i$ lies either fully below or above the $y$-range of $A_j$, depending on whether $Y_i < Y_{i'}$ or $Y_i > Y_{i'}$. 
	
	Thus, the order relation of $\mu_\cG$-samples corresponding to different indices from  $[r]$ is the same as for $(X_1, Y_1), \ldots, (X_r, Y_r)$. That is, it is distributed like $\mathrm{Perm}(r, \mu_{\cH})$. 
	
	By the definition of $A_i$, $i \ge 1$, it also follows that the order relation between $\mu_\cG$-samples corresponding to the same index from $[r]$ is identical to the order relation between the corresponding instance of $\mu_i$, i.e. it is distributed like $\mathrm{Perm}(k_i, \mu_i)$.
	
	Therefore, the permutation $\mathrm{Perm}(k, \mu_{\cG})$ induced by the $k$ independent samples of $\mu_\cG$ is distributed as the substitution $\nu[\nu_1,\ldots,\nu_r]$.
\end{proof}

In the general setting considered here, the random permuton $\mu_{\cG}$ is the distributional limit of the random permutations $\sigma^\cG_n$ as their size $n$ tends to infinity.

\begin{theorem}
	\label{te:main3}
	Suppose that Conditions \eqref{eq:pgraphcrit}--\eqref{eq:pcompconv} hold. Then
	\[
	\sigma_n^{\cG} \convd \mu_{\mathrm{PD}}(\alpha, -\alpha \beta, L_{\cH}, L_{\cC})
	\]
	as random permutons.
\end{theorem}
\begin{proof}
	By analogous arguments as for Theorem~\ref{te:main1}, it follows that the pattern of $\sigma_n^{\cG}$ induced by $k \ge 1$ random distinct samples converges in distribution to the random permutation $\mathrm{Perm}(k, \mu_{\cG})$ described in Lemma~\ref{le:finpdperm}. Since this holds for every fixed $k$, Proposition~\ref{pro:doit} yields
	 $\sigma_n^\cG\convd\mu_\cG$ as random permutons.
\end{proof}

\subsection{Repellent regimes}

In the dilute regime the asymptotic shapes of the head structure and component structures combine into a new limit object. We describe a phase diagram for regimes where the limit is with some probability identical to that of the head or identical to that of a single large component instead.

\subsubsection{Dense regime}

In the  dense regime  the head structure dominates: Suppose that $\rho_\cH>0$ and
\[
h_n = \ell_\cH(n) n^{-b}\rho_\cH^{-n} 
\]
for a slowly varying function $\ell_\cH$ and an exponent $b>1$. Furthermore suppose that one of the following  cases holds.
\begin{enumerate}[\qquad i)]
	\item 	We have $\rho_\cH = \cC(\rho_\cC)$ and 
	\[
	c_n = \ell_\cC(n) n^{- a} \rho_\cC^{-n}
	\]	
	for a slowly varying function  $\ell_\cC$ and an exponent $a>2$, such that  $1 < b <a$, or  $b=a$ and  $\ell_\cC(n) = o(\ell_\cH(n))$.
	\item We have $\rho_\cH < \cC(\rho_\cC)$ and $\gcd\{ n \ge 0 \mid c_n > 0\} = 1$.
\end{enumerate}
Moreover, assume that
\[
\sigma_n^\cH \convd \mu_\cH
\]
for a random permuton $\mu_\cH$.

\begin{theorem}
	\label{te:perdense}
	Under the preceding assumptions, we have $\sigma_n^\cG \convd \mu_\cH$ as $n \to \infty$.
\end{theorem}
\begin{proof}
	The proof is analogous to the proof of Theorem~\ref{te:dense}.
\end{proof}

\subsubsection{Condensation regimes}

In the condensation regime a single giant component dominates: We assume that there exists a random permuton  $\mu_\cC$ such that  
\begin{align*}
	\sigma_n^\cC \convd \mu_{\cC}.
\end{align*}
There are three subcases:

\vspace{1em}

\noindent \textbf{Convergent case:}
The convergent regime assumes that $0<\cH'(\cC(\rho_\cC))< \infty$ and any of the following conditions hold.

\begin{enumerate}[\qquad i)]
	\item We have $\rho_\cH = \cC(\rho_\cC)$ and \[
	h_n = \ell_\cH(n) n^{-b}\rho_\cH^{-n} \qquad \text{and} \qquad c_n = \ell_\cC(n) n^{- a} \rho_\cC^{-n}
	\]
	for slowly varying functions~$\ell_\cH$ and $\ell_\cC$, an exponent $b > 2$, and an exponent~$a$ such that either $1 < a < b$, or $a=b$ and $\ell_\cH(n) = o(\ell_\cC(n))$.
	\item We have $\infty \ge \rho_\cH > \cC(\rho_\cC)$ and \[
	\frac{c_n}{c_{n+1}} \to \rho_\cC > 0 \qquad \text{and} \qquad \frac{1}{c_n} \sum_{i+j = n} c_i c_j \to 2 \cC(\rho_\cC) < \infty\] as $n \to \infty$. (Note that the last two limits are automatically satisfied in case $c_n = \ell_\cC(n) n^{- a} \rho_\cC^{-n}$ with $a>1$ constant,  $\ell_\cC$ slowly varying, and $\rho_\cC > 0$.)
	\item We have $\sum_{k \ge 0} h_k \rho_{\cH}^k k^{1+a+\delta} < \infty$ and $c_n = \ell_\cC(n) n^{- a} \rho_\cC^{-n}$ for $a>1$, $\delta>0$ constants  and $\ell_\cC$ a slowly varying function.
	\item We have $c_n \sim c_\cC n^{- a} \rho_\cC^{-n}$ for constants $a>1$, $c_\cC > 0$, and additionally one of the following conditions hold.
	\subitem a) $\sum_{k \ge 0} h_k \rho_{\cH}^k k^{1+a} < \infty$.
	\subitem b) $1<a<2$.
	\subitem c) $a = 2$ and $\sum_{k \ge 0} h_k \rho_{\cH}^k k (\log k)^{2 + \delta} < \infty$ for some $\delta >0$.
	\subitem d) $2<a<3$ and $h_n \rho_{\cH}^n = o(n^{-a})$.
	\subitem e) $a=3$ and $h_n \rho_{\cH}^n = o(1 / (n^3 \log \log n))$.
\end{enumerate}

These conditions originate from the study of randomly stopped sums~\cite{MR4038060} and composition schemes with subexponential densities~\cite{MR0348393,MR772907}.

\vspace{1em}

\noindent\textbf{Superexponential case:} In this regime we assume $\rho_\cH \in ]0,\infty]$ and $\rho_\cC= 0$. Furthermore, assume that for any $k \ge 2$ we have
\[
\sum_{\substack{i_1 + \ldots i_k = n \\ 1 \le i_1, \ldots i_k < n - (k-1)}} c_{i_1} \cdots c_{i_k} = o(c_{n-(k-1)}).
\]
A sufficient condition for this equation is given~\cite[Lem. 6.17]{MR4132643}. We furthermore assume that $c_i>0$ for all large enough integers $i$.

\vspace{1em}

\noindent\textbf{Mesocondensation case: }  In this regime, we assume $\rho_\cH = \cC(\rho_\cC)$ and \[
h_n = \ell_\cH(n) n^{-b}\rho_\cH^{-n} \qquad \text{and} \qquad c_n = \ell_\cC(n) n^{- 1} \rho_\cC^{-n}
\]
for slowly varying functions~$\ell_\cH$ and $\ell_\cC$, and an exponent $b > -2$.

\vspace{1em}

\begin{theorem}
	\label{te:perdon}
	In any of the three preceding cases, we have $\sigma_n^\cG \convd \mu_\cC$ as $n \to \infty$.
\end{theorem}
\begin{proof}
	Analogous to the proof of Theorem~\ref{te:condensation}.
\end{proof}

\subsubsection{Mixture regime}
In the mixture regime the limit of $\sigma_n^\cG$ is a mixture of the limits of large head structures and large components.

Suppose that $\rho_\cH = \cC(\rho_\cC)$. Furthermore, suppose that
\[
h_n = \ell_\cH(n) n^{-a}\rho_\cH^{-n} \qquad \text{and} \qquad c_n = \ell_\cC(n) n^{- a} \rho_\cC^{-n}
\]	
for slowly varying functions~$\ell_\cH, \ell_\cC$ and a constant $a > 2$. Let 
\[
\mu = \cC(\rho_\cC)^{-1} \sum_{n \ge 1 } n c_n \rho_\cC^n.
\] 
Suppose that the limit
\[
p := \lim_{n \to \infty} \frac{\mu^{a-1}}{\cH'(\cC(\rho_\cC))}  \frac{\ell_\cH(n)}{\ell_\cC(n)} >0
\]
exists and is positive. Suppose that there exist random permutons $\mu_\cH$ and $\mu_\cC$ with
\begin{align*}
	\sigma_n^\cH \convd \mu_{\cH} \qquad \text{and} \qquad \sigma_n^\cC \convd \mu_{\cC}.
\end{align*}

\begin{theorem}
	Under the preceding assumptions, we have $0<p<1$ and \[
	\sigma_n^\cG \convd B \mu_\cH + (1-B) \mu_{\cC}
	\] with $B$ denoting an independent $\mathrm{Bernoulli}(p)$ random variable.
\end{theorem}
\begin{proof}
	Analogous to the proof of Theorem~\ref{te:mixture}.
\end{proof}

\section{Separable superpermutations}

\label{sec:supp}

\subsection{Background on separable permutations}

We recall background on the class $\cP$ of separable permutations. It is known as the class of all permutations avoiding the patterns 2413 and 3142. Alternatively, it is the smallest class of all permutations that is closed under taking substitutions of one member with another and contains all permutations from the symmetric groups $\cS_1$ and $\cS_2$~\cite{MR1620935}.

Separable permutations may be encoded using decomposition trees. We may define a \emph{generalized decomposition tree} as a plane tree where inner vertices carry signs ($\oplus$ or $\ominus$) and have at least two children. The separable permutation corresponding to such a tree is formed by taking the number of leaves as the size of the permutation, and setting $i \prec j$ if and only if the lowest common ancestor of the $i$th and $j$th leaves in depth-first-search order has a $\oplus$-sign. This yields a linear order that corresponds to a permutation which maps an integer to the corresponding position in this order. There may be multiple generalized decomposition trees corresponding to the same permutation, hence for \emph{decomposition trees} we additionally require that the sign of any non-root vertex is the opposite of the sign of its parent. Hence the sign of the root vertex determines all other signs. Any separable permutation may be represented in a unique manner by a decomposition tree~\cite{MR1620935}.

Decomposing according to the degree of the root vertex yields that class $\cA$ of decomposition trees whose root-vertex has a plus sign satisfies 
\begin{align}
	\label{eq:asum}
	\cA(z) &= z + \sum_{k \ge 2} \cA(z)^k   \\
		&= z + \cA(z)^2 / (1 - \cA(z)) \nonumber  \\
		&= \frac{1}{4}(1 +z - \sqrt{1 - 6z +z^2}). \nonumber
\end{align}
By symmetry,
\begin{align}
	\cP(z) = 2 \cA(z) - z.
\end{align}
It is elementary that $\cP(z)$ has radius of convergence
\begin{align}
	\label{eq:dadyaa}
	\rho_\cP= 3 - 2 \sqrt{2}
\end{align}
and
\begin{align}
	\label{eq:dadyan}
	\cP(\rho_\cP) = \sqrt{2}-1
\end{align}
and
\begin{align}
	\label{eq:apar}
	[z^n] \cP(z) \sim  \frac{\sqrt{2}-1}{2^{3/4}\sqrt{\pi}} n^{-3/2} \rho_\cP^{-n}
\end{align}
as $n \to \infty$.

From Equation~\eqref{eq:apar} it follows by a general principle~\cite[Lem 6.7]{MR4132643} (which in this case may also be easily verified directly), that a uniform decomposition tree with $n$ leaves is distributed like a Bienaym\'e--Galton--Watson tree conditioned on having $n$ leaves. The offspring distribution $\xi$ is given by 
\begin{align*}
	\Pr{\xi = 0} = 2 - \sqrt{2}, \quad \Pr{\xi=1} = 0, \quad \Pr{\xi = k} = \left( (2 - \sqrt{2})/2 \right)^{k-1} \quad \text{for $k \ge 2$}.
\end{align*}
Clearly $\Ex{\xi}=1$ and  $0< \Va{\xi}<\infty$.
The sign of the root vertex is determined by a fair coin flip, and this determines the signs of all other inner vertices. 

The  uniform random separable permutation $\sigma_n^\cP$ of size $n$  is distributed like the permutation corresponding to this decomposition tree. By Equation~\eqref{eq:yooo} it follows that if we select $k$ uniform independent leaves in the decomposition tree, then the corresponding essential vertices together with the $\oplus$/$\ominus$ signs converge in distribution to the uniform random $k$-tree $\mS(k)$ with  signs on the internal vertices determined by independent fair coin-flips. Let $\sigma(k)$ denote the corresponding permutation. It follows that the pattern $\mathrm{pat}_{I_{n,k}}(\sigma_n^\cP)$ induced by a uniformly chosen $k$-element  subset $I_{n,k} \subset [n]$ satisfies
\[
\mathrm{pat}_{I_{n,k}}(\sigma_n^\cP) \convd \sigma(k).
\]
By Proposition~\ref{pro:doit}, this entails that 
\begin{align}
	\label{eq:mupermuton}
	\sigma_n^\cP \convd \mu_{\mathrm{Brownian}}
\end{align}
for some random permuton $\mu_{\mathrm{Brownian}}$. This was shown in~\cite{zbMATH06919021} by an argument that involves decomposition trees but analyses their asymptotic shape combinatorially instead of using Aldous' local limit theorem~\cite[Lem. 26]{MR1207226} for the finite dimensional marginals. The short argument given here is closer to that of~\cite{MR4115736}. The random permuton $\mu_{\mathrm{Brownian}}$ is called the Brownian (separable) permuton, as it may be constructed from a  Brownian excursion~\cite{zbMATH07359153}. Further notable universal limiting objects include the Baxter permuton~\cite{zbMATH07527828},       the Brownian skew permuton~\cite{zbMATH07740462}, square permutons~\cite{zbMATH07325638},  the recursive separable permuton~\cite{zbMATH07836948}, and the anti-diagonal~\cite{hohmeier2025permutonlimitspermutationsavoiding}. 

\subsection{Multidimensional separable superpermutations with splits}

We set $\cP_1 = \cP$ to the class of separable permutations. For each $d \ge 2$ we let $\cP_d$ denote the class of superpermutations with head structures in $\cP$ and component structures from the class of ordered pairs $(P, P')$ of superpermutations from $\cP_{d-1}$. The ordered pair $(P,P')$ is identified with a superpermutation via the substitution $\mathrm{id}_2[P,P']$, with $\mathrm{id}_2 \in \cS_2$ denoting the identity permutation of size $2$. This way,
\[
	\cP_{d}(z) = \cP(\cP_{d-1}(z)^2).
\]
Due to the ``split'' of the components into two parts we refer to $\cP_d$ as the class of $d$-dimensional separable superpermutations with splits.

\begin{lemma}
	\label{le:sepasym}
	For each integer $d \ge 2$ we have 
	\[
	[z^n] \cP_d(z) \sim (\sqrt{2}-1)\frac{2^{3/2-d-5/2^{d+1}}}{\Gamma(1-2^{-d})}n^{-1-2^{-d}} \rho_\cP^{-n}
	\]
	with $\rho_\cP= 3- 2 \sqrt{2}$ and $\cP_d(\rho_\cP)^2 = \rho_\cP$.
\end{lemma}
\begin{proof}
		 Equation~\eqref{eq:dadyan} yields \[
		 \cP_1(\rho_\cP)^2= (\sqrt{2}-1)^2 = 3 - 2 \sqrt{2} = \rho_\cP.
		 \] For $d \ge 2$, let us assume by induction that $\cP_{d-1}^2(\rho_\cP)=\rho_\cP$ and that  
		 \[
		 [z^{n}]\cP_{d-1}(z) \sim C_{d-1} n^{-1-2^{-(d-1)}} \rho_\cP^{-n}
		 \] 
		for a constant $C_{d-1}>0$, with $C_1 =  \frac{\sqrt{2}-1}{2^{3/4}\sqrt{\pi}}$ by~\eqref{eq:apar}.
		 Then
		 \[
		 	\cP_d(\rho_\cP)^2 = \cP(\rho_\cP)^2 = \rho_\cP
		 \]
		 and
		 \begin{align*}
		 	[z^n]\cP_{d-1}(z)^2 &\sim 2\cP_{d-1}(\rho_\cP) [z^n] \cP_{d-1}(z) \\
		 						&= 2 \sqrt{\rho_\cP} [z^n] \cP_{d-1}(z)
		 \end{align*}
		 by standard methods, see for example~\cite[Thm. 2.4]{MR3854044}.  Hence we may apply Equation~\eqref{eq:partitionfucntion} for $U(z) = \cP_d(z)$,  $V(z) = \cP(z)$ and $W(z) = \cP_{d-1}(z)^2$. Specifically, $\rho_v=\rho_u=\rho_w = \rho_\cP$, $\beta=1/2$, $\alpha = 2^{-(d-1)}$, and 
	\begin{align*}
		L_v(x) &\sim C_1, \\
		L_w(x) &\sim 2 \sqrt{\rho_\cP} C_{d-1} / \alpha, \\
		A(x) &\sim x^\alpha / ( \rho_\cP^{-1} L_w(x) ).
	\end{align*}
	Hence, by~\eqref{eq:partitionfucntion} and~\eqref{eq:xalphadilute}
	\begin{align*}
		[z^n] \cP_d(z) &\sim L_v(A(n)) A(n)^{-\beta} n^{-1} \alpha \Ex{X_{\alpha}^{\alpha\beta}} \rho_\cP^{-n} \\ 
		&\sim C_{1} (n^\alpha / ( \rho_\cP^{-1} L_w(n) ))^{- \beta}    n^{-1} \alpha \Ex{X_{\alpha}^{\alpha\beta}} \rho_\cP^{-n} \\
		&\sim C_{1} ( n^\alpha / ( \rho_\cP^{-1} 2 \sqrt{\rho_\cP} C_{d-1} / \alpha ) )^{- \beta}    n^{-1} \alpha \Ex{X_{\alpha}^{\alpha\beta}} \rho_\cP^{-n} \\
		&\sim C_d n ^{-1-2^{-d}} \rho_\cP^{-n} 
	\end{align*}
	for
	\begin{align*}
		C_d &= C_1 \left( \frac{2 C_{d-1}}{\alpha \sqrt{\rho_\cP}}   \right)^{\beta} \alpha \Exb{X_\alpha^{\alpha \beta}}  \\
		&= C_1 \left( \frac{2 C_{d-1}}{\alpha \sqrt{\rho_\cP}}   \right)^{\beta} \alpha \Gamma(1-\alpha)^{\beta} \frac{\Gamma(1- \beta)}{\Gamma(1-\alpha \beta)} \\
		&= C_1 \left( \frac{2 C_{d-1}}{2^{-(d-1)} (\sqrt{2}-1)}   \right)^{1/2} 2^{-(d-1)} \Gamma(1-2^{-(d-1)})^{1/2} \frac{\sqrt{\pi}}{\Gamma(1-2^{-d})} \\ 
		&= \frac{\sqrt{2}-1}{2^{3/4}\sqrt{\pi}} \left( \frac{2^d C_{d-1}}{ \sqrt{2}-1}   \right)^{1/2} 2^{-(d-1)} \Gamma(1-2^{-(d-1)})^{1/2} \frac{\sqrt{\pi}}{\Gamma(1-2^{-d})} \\
		&= 2^{1/4 - d/2} (\sqrt{2}-1)^{1/2} \sqrt{C_{d-1}} \frac{\Gamma(1-2^{-(d-1)})^{1/2}}{\Gamma(1-2^{-d})}.
	\end{align*}
	Solving the recursion yields
	\[
		C_{d} = (\sqrt{2}-1)\frac{2^{3/2-d-5/2^{d+1}}}{\Gamma(1-2^{-d})},
	\]
	which is also correct for $d=1$.
	This completes the proof.
\end{proof}

We let $\sigma_{n,d}$ denote the uniform random $n$-sized $d$-dimensional separable superpermutation with splits. Let $L_{\mathrm{Brownian}} = \mathfrak{L}(\mu_{\mathrm{Brownian}})$ denote the law of the Brownian (separable) permuton.

\begin{theorem}
	\label{te:main4}
	Let $d \ge 2$ denote an integer. As $n\to \infty$,
	\[
	\sigma_{n,d} \convd \mu^{(d)} := \mu_{\mathrm{PD}}(2^{-(d-1)}, -2^{-d}, L_{\mathrm{Brownian}}, L_{d-1})
	\]
	as random permutons. Here we set $L_1 = L_{\mathrm{Brownian}}$, and recursively 
	$
	L_{d-1} = \mathfrak{L}(\mu^{(d-1)} )
	$
	the law of the limit for $d-1$. 
\end{theorem}
\begin{proof}
	Using Lemma~\ref{le:sepasym} it follows that we are in the dilute  regime for  $U(z) = \cP_d(z)$,  $V(z) = \cP(z)$ and $W(z) = \cP_{d-1}(z)^2$ and hence $\beta=1/2$ and $\alpha= 2^{-(d-1)}$. Thus $\theta= -2^{-d}$.

	For $d=2$, Equation~\eqref{eq:mupermuton} yields permuton convergence of the head structure towards the Brownian permuton. As for the components, it follows from~\cite[Thm. 3.1]{MR3854044} that in a random permutation of size $n$ from the class $\cP^2$ of pairs of separable permutation with limiting probability $1/2$ the first component has size $n- O_p(1)$ (and the second $O_p(1)$, and with limiting probability $1/2$ the second component has $n-O_p(1)$ (and hence the first has size $O_p(1)$). Consequently, the permuton limit of a large component from $\cP^2$ is the Brownian permuton as well.
	
	By Theorem~\ref{te:main3}, it follows that
	\[
	\sigma_{n,2} \convd \mu^{(2)} := \mu_{\mathrm{PD}}(2^{-1}, -2^{-2}, L_{\mathrm{Brownian}}, L_{\mathrm{Brownian}}).
	\]
	By Induction on $d$ and Theorem~\ref{te:main3}, we obtain analogously for all $d \ge 2$
	\[
	\sigma_{n,d} \convd \mu^{(d)} := \mu_{\mathrm{PD}}(2^{-(d-1)}, -2^{-d}, L_{\mathrm{Brownian}}, L_{d-1}).
	\]
	This completes the proof.
\end{proof}

We set $\mu^{(1)} = \mu_{\mathrm{Brownian}}$ to  keep the notation consistent. The random permuton in Theorem~\ref{te:main4} admits an alternative parametrization:

\begin{proposition}
	\label{pro:permchar}
	As random permutons,
	\[
	\mu^{(d)}   \eqdist \mu_{\mathrm{PD}}(1/2, -2^{-d},  L_{d-1}, L_{\mathrm{Brownian}}).
	\]
\end{proposition}
\begin{proof}
	Lemma~\ref{le:finpdperm} and Pitman's coagulation-fragmentation duality~\eqref{eq:duality} allow us to argue analogously as in the proof of Proposition~\ref{pro:char}. 
\end{proof}

\subsection{Weighted separable superpermutations}

Given a parameter $q>0$ we consider two-dimensional separable superpermutations where each component receives weight $q$. That is, we consider the class $\cP^{\langle q \rangle} = \cH \circ \cC$ where $\cH = \cP$ with the trivial weighting that assigns weight $1$ to each permutation, and $\cC = \cP$ with the weighting that assigns weight $q$ to each permutation. Hence the weight of a permutation from $\cP^{\langle q \rangle}$ equals $q$ raised to  the number of components, or equivalently, the size of the head-structure. Hence
\[
\cP^{\langle q \rangle}(z) = \cP(q\cP(z)).
\]
Let $\sigma^{\langle q \rangle}_n$ be drawn from all superpermutations of size $n$ in $\cP^{\langle q \rangle}$ with probability proportional to its weight.

\begin{corollary}
	As $n \to \infty$,
	\[
	\sigma^{\langle q \rangle}_n \convd \begin{cases}
		\mu_{\mathrm{Brownian}}, &q \neq \sqrt{2}-1 \\
		\mu^{(2)}, &q = \sqrt{2}-1
	\end{cases}
	\]
	as random permutons.
\end{corollary}
\begin{proof}
	By Equations~\eqref{eq:dadyaa},~\eqref{eq:dadyan} and~\eqref{eq:apar} we are in the dilute regime for $q = \sqrt{2}-1$, and
	\[
	\sigma^{\langle q \rangle}_n \convd \mu^{(2)} = \mu_{\mathrm{PD}}(1/2, -1/4, L_{\mathrm{Brownian}}, L_{\mathrm{Brownian}})
	\]
	by Theorem~\ref{te:main3}. For $0<q<\sqrt{2}-1$ we are the  condensation regime (specifically the convergent case) and  Theorem~\ref{te:perdon} yields
	\[
	\sigma^{\langle q \rangle}_n \convd \mu_{\mathrm{Brownian}}	\]
	For $q> \sqrt{2}-1$ we are in the dense regime and  Theorem~\ref{te:perdense} yields
	\[
	\sigma^{\langle q \rangle}_n \convd \mu_{\mathrm{Brownian}}.
	\]
\end{proof}

\bibliographystyle{abbrv}
\bibliography{pdpg}

\end{document}